\documentclass{amsart}

\calclayout

\usepackage[unicode]{hyperref}
\usepackage[lite]{amsrefs}
\usepackage{amsmath,amssymb}
\let\zeroslash\emptyset

\usepackage{tikz-cd}
\def\paragraph{\subsection}
\makeatletter
\def\@sect#1#2#3#4#5#6[#7]#8{%
  \edef\@toclevel{\ifnum#2=\@m 0\else\number#2\fi}%
  \ifnum #2>\c@secnumdepth \let\@secnumber\@empty
  \else \@xp\let\@xp\@secnumber\csname the#1\endcsname\fi
  \@tempskipa #5\relax
  \ifnum #2>\c@secnumdepth
    \let\@svsec\@empty
  \else
    \refstepcounter{#1}%
    \edef\@secnumpunct{%
      \ifdim\@tempskipa>\z@ % not a run-in section heading
        \@ifnotempty{#8}{.\@nx\enspace}%
      \else
        \@ifempty{#8}{.}{.\@nx\enspace}%
      \fi
    }%
      \ifnum #2=\tw@ \def\@secnumfont{\bfseries}\fi{}%
    \protected@edef\@svsec{%
      \ifnum#2<\@m
        \@ifundefined{#1name}{}{%
          \ignorespaces\csname #1name\endcsname\space
        }%
      \fi
      \@seccntformat{#1}%
    }%
  \fi
  \ifdim \@tempskipa>\z@ % then this is not a run-in section heading
    \begingroup #6\relax
    \@hangfrom{\hskip #3\relax\@svsec}{\interlinepenalty\@M #8\par}%
    \endgroup
    \ifnum#2>\@m \else \@tocwrite{#1}{#8}\fi
  \else
  \def\@svsechd{#6\hskip #3\@svsec
    \@ifnotempty{#8}{\ignorespaces#8\unskip
       \@addpunct.}%
    \ifnum#2>\@m \else \@tocwrite{#1}{#8}\fi
  }%
  \fi
  \global\@nobreaktrue
  \@xsect{#5}}
\makeatother

\newcommand{\hatzero}{\hat{0}}

\newcommand{\TNall}{\textnormal{ all }}
\newcommand{\TNand}{\textnormal{ and }}
\newcommand{\TNfor}{\textnormal{ for }}
\newcommand{\TNimplies}{\textnormal{ implies }}

\newtheorem{theorem}{Theorem}[section]

\newtheorem{definition}[theorem]{Definition}

\newtheorem{remark}[theorem]{Remark}

\newtheorem{question}[theorem]{Question}

\begin{document}

\title[Survey of lattice properties]{A survey of lattice properties:
  modular, Arguesian, linear, and distributive}
\author{Dale R. Worley}
\email{worley@alum.mit.edu}
\date{Apr 11, 2024} % format is Mmm dd, yyyy.

\begin{abstract}
This is a survey of characterizations and relationships between some
properties of lattices, particularly the modular, Arguesian, linear,
and distributive properties, but also some other related properties.
The survey emphasizes finite and finitary lattices and deemphasizes
complemented lattices.

A final section is a restatement of the open questions, which may
prove to be a source of thesis problems.
\end{abstract}

\maketitle

%\tableofcontents
%\listoffigures

\begin{quotation}
When an author collects together the opinions of as many others as he can
and fills half of every page with footnotes, this is known as
``scholarship''. --- Egyptologist Flinders Petrie quoted
in \cite{Bell1961}
\end{quotation}

% Push the note to the reader down a bit.
\vspace{2em}

\textit{Please send any suggestions for improvements,
particularly to the ``unknown'' items in section~\ref{sec:summary-chart},
to the author for
incorporation into a later version.}

\section{Introduction}

This is a survey of characterizations and relationships between some
properties of lattices, particularly the modular, Arguesian, linear,
and distributive properties, but also some other related properties.
We also enumerate how authors spell ``Arguesian'' in order to find the
consensus.

Because the present author's interest is in lattices that are
finitary\footnote{all principal ideals are finite}, and
which are not complemented, or even relatively complemented,
the survey emphasizes finite and finitary lattices and deemphasizes
complemented lattices.

This survey is organized into four sections.  The first
section \ref{sec:by-work} is organized by the works surveyed, quoting
the relevant results.  The second section \ref{sec:by-property} is
organized by the properties involved, synthesizing the results quoted
in the first part.  The third section \ref{sec:summary-chart} is a
summary chart of the lattice properties, the relationships between
them and their major characteristics.  The fourth
section \ref{sec:summary-qs} is a restatement of the open questions
from the second section, which may prove to be a source of thesis problems.

\section{Survey by works cited} \label{sec:by-work}

Words quoted with ``...'' are taken directly from the cited work.
Words that are not explicitly quoted are the present author's
paraphrase.  Quotations that continue into later paragraphs start each
later paragraph with ``.
All translations are the present author's.  Quoted words
are sometimes edited for consistency of terminology, brevity, or
especially, to make citations match the present bibliography.
Within quotations, words
bracketed [...] are the present author's paraphrase.  Deleted passages
are represented with ... .
For the exact wordings, please see the cited works themselves.

\paragraph{Birkhoff, \textnormal{``On the structure of abstract algebras''}}
\cite{Birk1935}

\cite{Birk1935}*{sec.~21 Th.~22} ``Theorem 22: Every subgroup lattice
[a sublattice of subgroups of a group] is isomorphic with an
equivalence lattice [a sublattice of equivalence relations on a set],
and conversely.''

\cite{Birk1935}*{sec.~31}
``The preceding material suggests several interesting questions whose
answer is unknown.

``Some questions concern equivalence lattices. Is any lattice
realizable as a lattice of equivalence relations? Is the dual of any
equivalence lattice an equivalence lattice? ...
More generally, are equivalence
lattices a family in the species of algebras of double composition [a
variety of algebras with two operations].''

\paragraph{Birkhoff, \emph{Lattice theory}} \cite{Birk1967}

Uses ``Desarguesian'' but only applies it to projective
planes.\cite{Birk1967}*{ch.~I sec.~13}

\cite{Birk1967}*{ch.~I sec.~6 Th.~9} ``Theorem 9.  In any lattice, the
following identities are equivalent:
\begin{equation*}
x \wedge (y \vee z) = (x \wedge y) \vee (x \wedge z)
\TNall x, y, z \tag{L6'}
\end{equation*}
\begin{equation*}
x \vee (y \wedge z) = (x \vee y) \wedge (x \vee z)
\TNall x, y, z \tag{L6''} \text{ ''}
\end{equation*}

\cite{Birk1967}*{ch.~I sec.~6} ``A ring of sets is a family $\Psi$ of
subsets of a set $I$ which
contains with any two sets $S$ and $T$ also their (set-theoretic)
intersection $S \cap T$ and union $S \cup T$.  ...  Any ring of sets
under the natural ordering $S \subset T$ is a distributive
lattice.''

\cite{Birk1967}*{ch.~I sec.~7} ``the self-dual ``modular'' identity:
If $x \leq z$, then
$x \vee (y \wedge z) = (x \vee y) \wedge
z$.''

\cite{Birk1967}*{ch.~I sec.~7} ``Thus any distributive lattice
satisfies [the modular identity].''

\cite{Birk1967}*{ch.~I sec.~7 Th.~11} ``Theorem 11.  The normal
subgroups of any group $G$ form a modular
lattice.''

\cite{Birk1967}*{ch.~I sec.~7 Th.~12} ``Theorem 12.  Any nonmodular
lattice $L$ contains [the pentagon lattice $N_5$] as a sublattice.''

\cite{Birk1967}*{ch.~II sec.~7 Lem.~1} ``Lemma 1.  In a modular
lattice, $x=y$ is implied by the conditions
$a \wedge x = a \wedge y$ and $a \vee x = a \vee y$, provided that $x$
and $y$ are comparable (i.e., that $x \geq y$ or $y \geq
x$).''

\cite{Birk1967}*{ch.~II sec.~7 Lem.~3}
``Lemma 3.  In any modular lattice, for all $x, y, z$,
$[x \wedge (y \vee z)] \vee [y \wedge (z \vee x)] = (x \vee y) \wedge
(y \vee z) \wedge (z \vee x)$.''

\cite{Birk1967}*{ch.~II sec.~7 Th.~13}
``Theorem 13.  Any modular, nondistributive lattice $M$ contains a
sublattice isomorphic with [the diamond
lattice].''

\cite{Birk1967}*{ch.~II sec.~7 Th.~13 Cor.}
``Corollary.  For a lattice to be distributive, the following
condition is necessary and sufficient:  $a \wedge x = a \wedge y$ and
$a \vee x = a \vee y$ imply $x=y$.  That is, it is necessary and
sufficient that relative complements be unique.''

\cite{Birk1967}*{ch.~II sec.~8 Th.~14~Cor.}
``Corollary.  Any modular or semimodular poset of finite length is
graded by its height function $h[x]$.''

\cite{Birk1967}*{ch.~II sec.~7 Exer.~2}
``Prove that, in any lattice, [the modular identity] is equivalent to
each of the identities $x \wedge (y \vee z) = x \wedge \{[y \wedge
(x \vee z)] \vee z\}$, $[(x \wedge z) \vee y] \wedge z = [(y \wedge
z) \vee x] \wedge z$.''

\cite{Birk1967}*{ch.~II sec.~8}
``We shall call $P$ (upper) semimodular when it satisfies:  If $a \neq
b$ both cover $c$, then there exists $d \in P$ which covers both $a$
and $b$.  Lower semimodular posets of finite lengths are defied
dually.''

\cite{Birk1967}*{ch.~II sec.~8 Th.~15}
``Theorem 15.  A graded lattice of finite length is semimodular if and
only if $h[x] + h[y] \geq h[x \vee y] + h[x \wedge y]$.
Dually, it is lower semimodular if and only if
$h[x] + h[y] \leq h[x \vee y] + h[x \wedge y]$.''
Thus a lattice of finite length is modular iff
it is graded and $h[x] + h[y] =
h[x \vee y] + h[x \wedge y]$.

\cite{Birk1967}*{ch.~II sec.~8 Th.~15 Cor.}
``Corollary.  In any modular lattice of finite length:
$h[x] + h[y] = h[x \vee y] + h[x \wedge
y]$.''

\cite{Birk1967}*{ch.~II sec.~8 Th.~16}
``Theorem 16.  Let $L$ be a lattice of finite length.  Then the
following conditions are equivalent:
\begin{enumerate}
\item[(i)] the modular identity L5,
\item[(ii)] $L$ is both upper and lower semimodular,
\item[(iii)] the Jordan--Dedekind chain condition [all maximal chains between
the same endpoints have the same finite length], and
[ $h[x] + h[y] = h[x \vee y] + h[x \wedge
y]$ ]. ''
\end{enumerate}

\cite{Birk1967}*{ch.~III sec.~3 Th.~3}
``Theorem 3.  Let $L$ be any distributive lattice of length $n$.  Then
the poset $X$ of join-irreducible elements $p_i > \hatzero$ has order
$n$ and $L = \boldsymbol{2}^X$ [the set of weakly order-preserving
functions from $X$ to $\boldsymbol{2}$].''

\paragraph{Britz, Mainetti, and Pezzoli,
\textnormal{``Some operations on the family of equivalence relations''}}
\cite{BritzMainPezz2001}

\cite{BritzMainPezz2001}*{Th.~1} ``Two equivalence relations $R$ and $T$ are
said to \emph{commute} if
and only if $RT = TR$.  Commuting equivalence relations can be
characterized by various means, among others by

``Theorem 1.  Let $R$ and $T$ be equivalence relations on a set $S$.
The following statements are then equivalent:
\begin{enumerate}
\item $R$ and $T$ commute;
\item $R \vee T = RT$;
\item $RT$ is an equivalence relation.''
\end{enumerate}

\cite{BritzMainPezz2001}*{Th.~2} ``Two equivalence relations $R$ and $T$ are
\emph{independent} if and only if $\rho \cap \tau$ is nonempty for all
classes $\rho \in S_R$ and $\tau \in S_T$.  The following
characterization in terms of the equivalence classes of $R$ and $T$ is
due to P. Dubreil and M.-L. Dubreil-Jacotin[3].

``Theorem 2.  The equivalence relations $R$ and $T$ commute if and
only if, for any equivalence class $C \in S_{R \vee T}$ 
the restrictions $R_{|C}$, $T_{|C}$ are independent.''

\paragraph{Crawley and Dilworth, \emph{Algebraic Theory of Lattices}}
\cite{CrawDil1973}

Uses ``Arguesian'' and ``has a representation of type 1''.

\cite{CrawDil1973}*{Chap.~12 12.4}
``12.4:  If a lattice $L$ has a representation of type 1, then $L$ is
Arguesian.\cite{Jons1953b}''

\paragraph{Day, \textnormal{``In search of a Pappian lattice identity''}}
\cite{Day1981}

\cite{Day1981}*{sec.~1}
``there is a lattice embedding of the non-pappian projective plane over
the quaternions into the 5-dimensional projective space over the complex
numbers.''

\paragraph{Day,
\textnormal{``Geometrical applications in modular lattices''}}
\cite{Day1982}

Uses ``Arguesian'' and ``Desarguean''.

\cite{Day1982}*{sec.~1}
``Theorem (Dedekind [27]). Let $(M; \vee, \wedge)$ be a lattice; then the
following are equivalent:
\begin{enumerate}
\item $(M; \vee, \wedge)$ is modular.
\item $(M; \vee, \wedge)$ satisfies $(\forall x, y, z \in M) (x \leq z \Rightarrow
x \vee (y \wedge z) = (x \vee  y) \wedge z)$.
\item $(M; \vee, \wedge)$ satisfies $(\forall x, y, z \in M)
((x \wedge z) \vee (y \wedge z) = ((x \wedge z) \vee y) \wedge
z))$.''
\end{enumerate}

Note there is a typo in the above reference; it should be ``[30]''.

\cite{Day1982}*{sec.~3 Def.~1}
``Definition 1. A lattice $(L; \vee, \wedge)$ is called Desarguean if
it satisfies the implication
$$ [(a_0 \vee b_0) \wedge (a_1 \vee b_1) \leq a_2 \vee b_2] \Rightarrow
[(a_0 \vee a_1) \wedge (b_0 \vee b_1) \leq
[(a_0 \vee a_2) \wedge (b_0 \vee b_2)] \vee
[(a_1 \vee a_2) \wedge (b_1 \vee b_2)] \text{. ''} $$

\cite{Day1982}*{sec.~3 Lem.~2}
``Lemma 2. Desarguean lattices are modular.''

\cite{Day1982}*{sec.~3 Th.~7}
``Theorem 7. A lattice is Desarguean if and only if it satisfies the
identity $\lambda \leq \rho$ where
\begin{align*}
\lambda & = (x_0 \vee x_1) \wedge (y_0^\prime \vee y_1) \\
\rho & = [(x_0 \vee x_2) \wedge (y_0 \vee y_2)] \vee
[(x_1 \vee x_2) \wedge (y_1 \vee y_2)] \vee [y_1 \wedge (x_0 \vee x_1)] \\
\intertext{and}
y_0^\prime & = y_0 \wedge [x_0 \vee ((x_1 \vee y_1) \wedge (x_2 \vee
y_2))] \text{. ''}
\end{align*}

\cite{Day1982}*{sec.~3 Th.~8}
``Let $a_0, a_1, a_2, b_0, b_1, b_2$ be six variables and define
lattice terms
$$ p = (a_0 \vee b_0) \wedge (a_1 \vee b_1) \wedge (a_2 \vee b_2) $$
$$ c_i = (a_j \vee a_k) \wedge (b_j \vee b_k)
\qquad \{i, j, k\} = \{0, 1, 2\} $$
and
$$ \bar{c} = c_2 \wedge (c_0 \vee c_1). $$

``Theorem 8. For a lattice $(L; \vee , \wedge)$, the following are equivalent.
\begin{enumerate}
\item $L$ is Desarguean
\item $L$ satisfies $p \leq a_0 \vee [b_0 \wedge (b_1 \vee \bar{c})]$
\item $L$ satisfies $p \leq [a_0 \wedge (a_1 \vee \bar{c})] \vee [b_0 \wedge (b_1 \vee \bar{c})]$
\item $L$ satisfies $p \leq a_0 \vee b_1 \vee \bar{c}$ ''
\end{enumerate}

\paragraph{Day and Pickering,
\textnormal{``A note on the Arguesian lattice identity''}} \cite{DayPick1984}

Uses ``Arguesian''.

\cite{DayPick1984} ``Let $(L; +, \cdot)$ be a
lattice. A \emph{triangle} in $L$ is an
element of $L^3$. For two triangles in $L$, $\boldsymbol{a} = (a_0,
a_1, a_2)$ and $\boldsymbol{b}=(b_0, b_1, b_2)$ we define auxiliary
polynomials $p_i=p_i(\boldsymbol{a},\boldsymbol{b})=(a_j+b_j)\cdot
(a_k+b_k)$, $p= p_ip_j\;(= p_ip_k=p_jp_k)$, and $c_i=c_i(\boldsymbol{a},
\boldsymbol{b}) = (a_j+ a_k) \cdot (b_j+b_k)$. Two triangles,
$\boldsymbol{a}$ and $\boldsymbol{b}$ in $L$, are called
\emph{centrally perspective} if $p_2(\boldsymbol{a},
\boldsymbol{b})\leq a_2+b_2$ and are called \emph{axially perspective}
if $c_2(\boldsymbol{a}, \boldsymbol{b})\leq c_0(\boldsymbol{a},
\boldsymbol{b})+ c_1(\boldsymbol{a}, \boldsymbol{b})$. We abbreviate
these concepts as $CP(\boldsymbol{a}, \boldsymbol{b})$ and
$AP(\boldsymbol{a}, \boldsymbol{b})$ respectively. \emph{Desargues’
implication} is the Horn sentence $CP(\boldsymbol{a}, \boldsymbol{b})
\Rightarrow AP(\boldsymbol{a}, \boldsymbol{b})$.

``Theorem. In the theory of lattices the following are equivalent.
\begin{enumerate}
\item Desargues’ Implication
\item $p(\boldsymbol{a}, \boldsymbol{b}) \leq a_0(a_1 + c_2(c_0+ c_1))
+ b_0(b_1+ c_2(c_0+ c_1))$
\item $p(\boldsymbol{a}, \boldsymbol{b}) \leq a_0+b_0(b_1 + c_2(c_0+c_1))$
\item $p(\boldsymbol{a},\boldsymbol{b}) \leq a_0+ b_1+ c_2(c_0+ c_1)$
\item $(a_0+ c_1) (b_0(a_0+p_0)+b_1) \leq c_0+c_1+b_1(a_0+ a_1)$.''
\end{enumerate}

\cite{DayPick1984} ``one might ask if [a self-dual] equation exists
for Arguesian lattices.''

\paragraph{Encyclopedia of Mathematics,
\textnormal{``Arguesian lattice''}} \cite{EnMathArg}

Uses ``Arguesian''.

\cite{EnMathArg} ``A lattice in which the Arguesian law is valid,
i.e. for all $a_i$, $b_i$,
$$ (a_0+b_0)(a_1+b_1)(a_2+b_2) \leq a_0(a_1+c)+b_0(b_1+c), $$
$c=c_0(c_1+c_2)$, $c_i=(a_j+a_k)(b_j+b_k)$ for any permutation $i$,
$j$, $k$ [a21].  Arguesian lattices form a variety [are characterized by a
set of identities], since within lattices $p \leq q$ is equivalent to
$pq=p$. A lattice is Arguesian if and only if it is a modular lattice
and $(a_0+b_0)(a_1+b_1) \leq a_2+b_2$ (central perspectivity) implies
$c_2 \leq c_0+c_1$ (axial perspectivity). In an Arguesian lattice and
for $a_i$, $b_i$ such that $a_2=(a_0+a_2)(a_1+a_2)$ and
$b_2=(b_0+b_2)(b_1+b_2)$, the converse implication is valid too [a24]. A
lattice is Arguesian if and only if its partial order dual is
Arguesian.''

\paragraph{Encyclopedia of Mathematics,
\textnormal{``Desargues assumption''}} \cite{EnMathDes}

\cite{EnMathDes} ``In lattice-theoretical terms the Desargues assumption
may be formulated as the identity ([1])
\begin{multline*}
[(x+z)(y+u)+(x+u)(y+z)](x+y) \\
\leq [(y+x)(z+u)+(y+u)(z+x)](y+z)+[(z+y)(x+u)+(z+u)(x+y)](z+x). \text{ ''}
\end{multline*}

Note that in light of para.~\ref{p:Haim1985b}\cite{Haim1985b}*{Note~2},
since this equation contains only four variables, it is likely
incorrect.

\paragraph{Encyclopedia of Mathematics,
\textnormal{``Modular lattice''}} \cite{EnMathMod}

\cite{EnMathMod} ``[A modular lattice is one] in which the modular law
is valid, i.e. if
$a \leq c$, then $(a+b)c=a+bc$ for any $b$.''

\cite{EnMathMod} ``[Modularity] amounts to saying that the identity
(ac+b)c=ac+bc is valid.''

\cite{EnMathMod} ``Examples of modular lattices include the lattices
of subspaces of a
linear space, of normal subgroups (but not all subgroups) of a group,
of ideals in a ring, etc.''

\cite{EnMathMod} ``A lattice with a composition sequence is a modular
lattice if and
only if there exists on it a dimension [rank] function $d$, i.e. an
integer-valued function such that $d(x+y)+d(xy)=d(x)+d(y)$ and such
that if the interval $[a,b]$ is prime [a covering], it follows that
$d(b)=d(a)+1$.''

\paragraph{Gr\"atzer, \emph{General lattice theory}} \cite{Gratz1996}

Uses ``Arguesian'' and ``representation of type 1''.

\cite{Gratz1996}*{Ch.~I sec.~4 Lem.~10}
``Lemma 10.  Consider the following two identities and inequality:
\begin{enumerate}
\item[(i)] $x \wedge (y \vee z) = (x \wedge y) \vee (x \wedge z)$,
\item[(ii)] $x \vee (y \wedge z) = (x \vee y) \wedge (x \vee z)$,
\item[(iii)] $(x \vee y) \wedge z \leq x \vee (y \wedge z)$.
\end{enumerate}
Then (i), (ii), and (iii) are equivalent in any lattice $L$.

``Remark.  A lattice satisfying identities (i) or (ii) is
called \emph{distributive}.''

\cite{Gratz1996}*{Ch.~I sec.~4 Lem.~12}
``Lemma 12.  The identity
$$ (x \wedge y) \vee (x \wedge z) = x \wedge (y \vee (x \wedge z)) $$
is equivalent to the condition:
$$ x \geq z \textnormal{ implies that } (x \wedge y) \vee z = x \wedge
(y \vee z). $$

``Remark.  A lattice satisfying either condition is
called \emph{modular}.''

\cite{Gratz1996}*{Ch.~IV sec.~1 Th.~1}
``Theorem 1.  For a lattice $L$, the following conditions are equivalent:
\begin{enumerate}
\item[(i)] $L$ is modular, that is,
$$ x \geq z \textnormal{ implies that } x \wedge (y \vee z) = (x \wedge y) \vee z. $$
\item[(ii)] $L$ satisfies the \emph{shearing identity}:
$$ x \wedge (y \vee z) = x \wedge ((y \wedge (x \vee z)) \vee z). $$
\item[(iii)] $L$ does not contain a pentagon.
\end{enumerate}

``...  Observe also the dual of the shearing identity:
$$ x \vee (y \wedge z) = x \vee ((y \vee (x \wedge z)) \wedge z) \text{. ''}. $$

\cite{Gratz1996}*{Ch.~IV sec.~1 Exer.~1}
``a lattice $L$ is modular iff it satisfies the identity
$$ (x \vee (y \wedge z)) \wedge (y \vee z) =
(x \wedge (y \vee z)) \vee (y \wedge z). $$

\cite{Gratz1996}*{Ch.~IV sec.~4 Def.~9}
``Definition 9.  Let $x_0$, $x_1$, $x_2$, $y_0$, $y_1$, $y_2$ be
variables.  We define some polynomials:
\begin{align*}
z_{ij} & = (x_i \vee x_j) \wedge (y_i \vee y_j), \qquad 0 \leq i < j < 3, \\
z & = z_{01} \wedge (z_{02} \vee z_{12}).
\end{align*}
The \emph{Arguesian identity} is
$$ (x_0 \vee y_0) \wedge (x_1 \vee y_1) \wedge (x_2 \vee y_2) \leq
((z \vee x_1) \wedge x_0) \vee ((z \vee y_1) \wedge y_0). $$

``A lattice satisfying this identity is called \emph{Arguesian}.''

\paragraph{Haiman,
\textnormal{``The theory of linear lattices''}}
\label{p:Haim1984a}\cite{Haim1984a}

Note most results of \cite{Haim1984a} are
in para.~\ref{p:Haim1985a}\cite{Haim1985a} and for many, we have not made
duplicate entries here.

Uses ``Arguesian'' and ``linear''.

\cite{Haim1984a}*{Intro.}
``the lattices of normal subgroups of a group, ideals of a ring, or
subspaces of a vector space are more than modular; as Birkhoff and
Dubreil-Jacotin were first to observe, they are lattices of
equivalence relations which commute relative to the operation of
composition of relations. The combinatorial properties of lattices of
commuting equivalence relations are not mere consequences of their
modularity, but rather the opposite; the consequences of the modular
law derived since Dedekind, who originally formulated it, have mainly
been guessed on the basis of
examples which were lattices of commuting equivalence relations.''

\cite{Haim1984a}*{Intro.}
``This thesis is a study of lattices of commuting
equivalence relations, which we have named \emph{linear} lattices,
a term suggested by Rota for its evocation of the archetypal
example of projective geometry.''

\cite{Haim1984a}*{Ch.~II}
``By general theorems of universal algebra (\cite{Gratz1979} Appendix 4,
Theorems 3, 4, for instance), the axiomatizability of the class of
linear lattices by universal equational Horn sentences, such as (1),
is equivalent to the closure of the class of linear lattices under
isomorphism, sublattices, and products. These closure properties can
also be seen directly. Namely, closure under isomorphism and
sublattices are built into the definition, and closure under products
is given by the sum of representations construction.''

\cite{Haim1984a}*{Ch.~II}
``From Theorem 1 it is clear that every implication (1) depends on
only finitely many of its hypotheses $P_i \leq Q_i$.  This fact implies
the further closure of the class of linear lattices under direct
limits. It does not seem possible, however, straightforwardly to
construct a representation of a direct limit of linear lattices from
representations of the individual lattices.''

\cite{Haim1984a}*{Ch.~II}
``It also does not seem possible, straightforwardly or otherwise, to
construct a linear representation of an arbitrary homomorphic image of
a linear lattice. To do so, of course, would show that the class of
linear lattices is a variety, answering the most important open
question concerning linear lattices.''

\paragraph{Haiman,
\textnormal{``Linear lattice proof theory: an overview''}} \cite{Haim1984b}

Uses ``Arguesian'' and ``linear''.

\cite{Haim1984b}*{sec.~1}
``Theorem 2, while not solving the free linear lattice word problem,
comes close and clearly isolates the difficulties. It also suggests a
natural conjecture [Conj.~2] which if true would simultaneously solve the free
linear lattice word problem and prove that the class of linear
lattices is not self-dual and (hence) not equal to the variety of
Arguesian lattices.''

\cite{Haim1984b}*{sec.~3}
``Conjecture 2. The number of deductions (A) required in a Normal Form
proof is zero.

``...

``As the next example shows, the truth of Conjecture 2 would imply that
the variety generated by linear lattices is not self-dual,
hence \cite{Jons1972} not equal to the variety of Arguesian lattices.

``Example 2. The lower series-parallel graph in Figure 4 is derived from
the upper one by using paired deductions (CD) to partition vertices $u$
and $v$ as indicated by the subscripted labels in the lower graph. The
corresponding lattice polynomials (which can be read from the graphs,
so there is no need to write them out in symbols) therefore are the
left- and right-hand sides of a valid linear lattice inequality. This
derivation is not planar, so it is not surprising to find (and
routine, though tedious, to verify) that the dual inequality,
displayed in Figure 5, has no proof using vertex partitioning and
deduction (E) alone. Both inequalities are short [in the sense defined
here], so if Conjecture 2 holds, the dual inequality cannot be valid
in all linear lattices.

``A linear lattice violating the inequality of Figure 5 would of course
be of interest, independent of Conjecture 2.''

\paragraph{Haiman
\textnormal{``Proof theory for linear lattices''}}
\label{p:Haim1985a}\cite{Haim1985a}

Most of these results are from para.~\ref{p:Haim1984a}\cite{Haim1984a}.

Uses ``Arguesian'' and ``linear''.

\cite{Haim1985a}*{sec.~0.0}
``the lattices of normal subgroups of a group, ideals of a ring,
or subspaces of a vector space are more than modular; as Birkhoff and
Dubreil-Jacotin were first to observe, they are lattices of equivalence
relations which commute relative to the operation of composition of
relations.''

\cite{Haim1985a}*{sec.~0.0}
``Let $p(S)$ denote the lattice of equivalence relations (or
partitions) on the set $S$. ...
If $\rho: L \rightarrow p(S)$ and $\rho^\prime: L^\prime
\rightarrow p(S^\prime)$ are linear representations with $S \cap
S^\prime = \zeroslash$, their \emph{sum} $\rho \oplus \rho^\prime: L
\times L^\prime \rightarrow p(S \cup S^\prime)$ defined by $\rho
\oplus \rho^\prime((x, x^\prime)) = \rho(x) \cup
\rho^\prime(x^\prime)$ is also a linear representation.  If $L=
L^\prime$ we also refer to $(\rho \oplus \rho^\prime) \circ \Delta: L
\rightarrow S \cup S^\prime$, where $\Delta(x) = (x, x)$ [the minimum,
discrete, or identity equivalence relation], as the
\emph{sum} of the representations $\rho$, $\rho^\prime$ of $L$.  Sums
of arbitrary finite or infinite collections of representations are
defined analogously.''

\cite{Haim1985a}*{sec.~1.0}
``Recall that the \emph{dual} of a lattice L is the lattice $L^d$ on
the same underlying set as $L$, but in which meet and join have been
interchanged, or what is the same, the partial order reversed. As is
well known, the varieties of distributive lattices, modular lattices,
and all lattices are closed under dualization. \cite{Jons1972} showed
that the variety of Arguesian lattices is also self-dual in this sense
(see (2.2))

``While it remains an open question whether the class of all linear
lattices is self-dual, there is an intrinsic duality to the list of
deductions (A)-(F).''

\cite{Haim1985a}*{Exam.~1.1}
``Example 1.1. The Arguesian implication \cite{Jons1954a}\cite{Jons1953b} is
$$ (a \vee a^\prime) \wedge (b \vee b^\prime) \leq c \vee c^\prime $$
implies
$$ (a \vee b) \wedge (a^\prime \vee b^\prime) \leq
[(a \vee c) \wedge (a^\prime \vee c^\prime)] \vee
[(c \vee b) \wedge (c^\prime \vee b^\prime)] \text{ \ldots} $$

``Reading the primed graphs
backward gives a proof of the dual implication
$$ c \wedge c^\prime \leq (a \wedge a^\prime) \vee (b \wedge b^\prime) $$
implies
$$ [(a \wedge c) \vee (a^\prime \wedge c^\prime)] \wedge
[(c \wedge b) \vee (c^\prime \wedge b^\prime)] \leq
(a \wedge b) \vee (a^\prime \wedge b^\prime) \text{. ''} $$
This proof of the
Arguesian implication shows that it is true for all linear
lattices.

\cite{Haim1985a}*{sec.~1.0}
``By general theorems of universal algebra ([15], Appendix 4,
Theorems 3, 4, for instance), the axiomatizability of the class of linear
lattices by universal equational Horn sentences, such as (1.1), is
equivalent to
the closure of the class of linear lattices under isomorphism, sublattices,
and products. These closure properties can also be seen directly. Namely,
closure under isomorphism and sublattices are built into the definition, and
closure under products is given by the sum of representations construction.

``From Theorem 1.1 it is clear that every implication (1.1) depends on
only finitely many of its hypotheses $P_i \leq Q_i$. This fact implies
the further
closure of the class of linear lattices under direct limits. It does not seem
possible, however, straightforwardly to construct a representation of a
direct limit of linear lattices from representations of the individual lattices.

``It also does not seem possible, straightforwardly or otherwise, to
construct a linear representation of an arbitrary homomorphic image of
a linear lattice. To do so, of course, would show that the class of
linear lattices is a variety, answering the most important open
question concerning linear lattices.''

\cite{Haim1985a}*{Exam.~2.2}
``Example 2.2. The Arguesian identity (\cite{Jons1953b} [7]---our
version is a simplified equivalent form \cite{Haim1985b}) is
$$ c \wedge ([(a \vee a^\prime) \wedge (b \vee b^\prime)] \vee c^\prime)
\leq a \vee ([((a \vee b) \wedge (a^\prime \vee b^\prime)) \vee
((b \vee c) \wedge (b^\prime \vee c^\prime))] \wedge (a^\prime \vee c^\prime)). $$
We now apply Proposition 2.1. Since every variable already occurs exactly
once on the left-hand side, the second subscript on the replacement
variables does not vary, and we suppress it. The equivalent form then reads
\begin{multline*}
c \wedge ([((a_1 \wedge a_2) \vee (a_1^\prime \wedge a_2^\prime))
\vee ((b_1 \wedge b_2) \vee (b_1^\prime \wedge b_2^\prime))]
\vee (c_1^\prime \wedge c_2^\prime)) \\
\leq a_1 \vee ([((a_2 \vee b_1) \wedge (a_2^\prime \vee b_1^\prime))
\vee ((b_2 \vee c) \wedge (b_2^\prime \vee c_1^\prime))] \wedge
(a_1 \vee c_2^\prime)). \tag{2.2}
\end{multline*}

``It is a remarkable fact that (2.2) becomes its own dual after
exchanging the symbols
$a_1 \leftrightarrow c$, $a_1^\prime \leftrightarrow c_1^\prime$,
$a_2 \leftrightarrow b_2$, and $a_2^\prime \leftrightarrow b_2^\prime$.
(2.2) is the first explicitly self-dual form of the Arguesian
identity to have been found.``

\cite{Haim1985a}*{Exam.~2.3}
``Example 2.3. One form of the modular law is the identity
$$ (b \vee c) \wedge a \leq b \vee (c \wedge (a \vee b)). $$
Proposition 2.1 converts this identity to
$$ ((b_1 \wedge b_2) \vee c) \wedge a \leq
b_1 \vee (c \wedge (a \vee b_2)) \text{. ''} $$

\cite{Haim1985a}*{Exam.~2.5, Figs.~6 and 7}
``Example 2.5. Generalizing Example 2.4 yields higher Arguesian
identities, the first of which is shown in Fig.~7. It has probably
occurred to the reader by now that a complicated lattice polynomial is
more easily recognized from a drawing of its series-parallel graph
then from an expression for the polynomial itself. Therefore we do not
actually write down a higher Arguesian identity, but let the graphs in
Fig.~7 stand for its left and right sides. The generalization beyond
Figs.~6 and 7 should be apparent. The proofs of the higher Arguesian
identities are all analogous to that of the usual Arguesian identity
(2.2); all are planar and self-dual.

``Each higher Arguesian identity is equivalent by Proposition 2.1 to
its special case in which $x_1 = x_2$ for each primed or unprimed
letter $x$ among the variables. Using this form and setting $a = b$,
$a^\prime = b^\prime$, one easily derives each identity from the next
higher one. It is not known whether each higher Arguesian identity is
strictly stronger than the ones below it. We suspect that this is
indeed the case.''

\cite{Haim1985a}*{sec.~3.0}``The four most fundamental open questions
about the class of linear lattices are (1) is it a variety, (2) is it
self-dual, (3) does it generate the variety of Arguesian lattices, and
(4) does it have a solvable free lattice word problem?  The
diagrammatic proof theory presented here has clear relevance for
attacks on the last three questions, as we have sought to make
clear.'''

\paragraph{Haiman,
\textnormal{``Two notes on the Arguesian identity''}}
\label{p:Haim1985b}\cite{Haim1985b}

Uses ``Arguesian''.

\cite{Haim1985b}*{eq.~1} This equation is equivalent to the Arguesian identity
$$ \text{`` } c \wedge ([(a \vee a^\prime) \wedge (b \vee b^\prime)] \vee  c^\prime)
\leq a \vee ([((a \vee b) \wedge (a^\prime \vee b^\prime)) \vee
((b \vee c) \wedge (b^\prime \vee c^\prime))] \wedge (a^\prime \vee c^\prime))
\text{ ''} $$

\cite{Haim1985b}*{eq.~2}  This equation is equivalent to the Arguesian identity
and is self-dual
\begin{multline*}
\text{`` } c \wedge ([((a_1 \wedge a_2) \vee (a_1^\prime \wedge a_2^\prime)) \wedge
((b_1 \wedge b_2) \vee (b_1^\prime \wedge b_2^\prime))] \vee
(c_1^\prime \wedge c_2^\prime)) \\
\leq a_1 \vee ([((c \vee b_2) \wedge (c_1^\prime \vee b_2^\prime)) \vee 
((b_1 \vee a_2) \wedge (b_1^\prime \vee a_2))] \wedge (a_1^\prime \vee c_2^\prime))
\text{ ''}
\end{multline*}

\cite{Haim1985b}*{Note~2} ``We construct a non-Arguesian lattice $L$ all
of whose [sublattices generated by five elements] are Arguesian; in
fact, are sublattices of a quaternionic projective geometry.''  Thus,
``any lattice identity equivalent to the Arguesian law must
necessarily involve at least six variables.''

\paragraph{Haiman,
\textnormal{``Arguesian lattices which are not linear''}} \cite{Haim1987}

Uses ``Arguesian'' and ``linear''.

\cite{Haim1987}*{sec.~1} ``Nevertheless, the question raised by J\'onsson,
whether every Arguesian lattice is linear, has remained open until
now.  Here we describe an infinite family ${A_n}$ $(n > 3)$ of
nonlinear lattices, Arguesian for $n > 7$ (and possibly for $n > 4$),
settling J\'onsson's question in the negative. Actually, we obtain
more: a specific infinite sequence of identities strictly between
Arguesian and linear, and a proof that the universal Horn theory of
linear lattices is not finitely based.''

\cite{Haim1987}*{sec.~3} ``Theorem. $A_n$ is not a linear lattice.

``Proof. In \cite{Haim1985a}, the author introduced ``higher Arguesian
identities''
$$ D_n\text{\@: }
a_0 \wedge \left(a_0^\prime \vee \bigwedge_{i=1}^{n-1} [a_i \vee a_i^\prime] \right)
\leq a_1 \vee \left( (a_0^\prime \vee a_1^\prime) \wedge
\bigvee_{i=1}^{n-1} [(a_i \vee a_{i+1}) \wedge (a_i^\prime \vee a_{i+1}^\prime)] \right)
$$
which hold in all linear lattices. $D_3$ is the Arguesian
law \cite{Haim1985b}. ...
$D_n$ fails in $A_n$ . In particular, $A_3$ is not Arguesian. This
minimally non-Arguesian lattice was discovered by Pickering [8].

``Theorem. Every proper sublattice of $A_n$ is linear.

``...

``Theorem. If $X \subset A_n$ generates $A_n$, then $|X| \geq n$.''

Note that in the above identity, indexes of $a$ and $a^\prime$ are to
be taken modulo $n$, that is, $a_0 = a_n$ and $a_0^\prime = a_n^\prime$.
Compare with equation~\ref{eq:higher}\cite{Hawr1996}*{eq.~15}.

\cite{Haim1987}*{sec.~4}
``The results of \S 3 imply that no finite set of identities,
or even universal Horn sentences, can completely characterize linearity; in
particular, the Arguesian law is insufficient, since it holds in $A_n$ for $n > 7$.
It is known, however, how to characterize linear lattices by an infinite set of
universal Horn sentences \cite{Haim1985a}\cite{Jons1959a}.''

\cite{Haim1987}{sec.~4}
``If, as appears likely, the identity $D_{n-1}$ holds in
$A_n$ $(n > 4)$, we would have that $D_{n-1}$ does not imply $D_n$,
showing that $\{D_n\}$ forms a hierarchy of progressively strictly
stronger linear lattice identities.''

\paragraph{Haiman,
\textnormal{``Arguesian lattices which are not type-1''}} \cite{Haim1991}

Uses ``Arguesian''.  Uses ``type-1 representability'' for lattices.
Uses ``linear representation'' for an isomorphism of a lattice with
a sublattice of the lattice of subspaces of a vector space.

\cite{Haim1991}*{sec.~1} ``It is known that there is an infinite
system of universal Horn sentences characterizing type-1
representability \cite{Haim1985a} \cite{Jons1959a}.''

\cite{Haim1991}*{sec.~1} ``Nation and Picketing \cite{NatPick1987}
conjectured that
having a distributive skeleton might be a sufficient condition for a
finite-dimensional Arguesian lattice to be type-1 representable.
[In section 4] we show how to modify the examples $A_n$, to give
counterexamples to this conjecture.''

\cite{Haim1991}*{sec.~3}
Uses the fact that if a lattice is a sublattice of the lattice of
subspaces of a vector space, then the lattice is Arguesian and linear.

\cite{Haim1991}*{sec.~3}
Constructs a series of lattices that are Arguesian but violate the
higher Arguesian identities which are satisfied by every linear
lattice.

\cite{Haim1991}*{sec.~3 Th.~1} ``Theorem 1.  Let $D_n$ be the lattice
identity
$$ x_0 \left( x_0^\prime + \prod_{i=1}^{n-1} [x_i + x_i^\prime] \right)
\leq x_1 + (x_0^\prime + x_1^\prime)
\sum_{i=1}^{n-1} (x_i + x_{i+1})(x_i^\prime + x_{i+1}^\prime). $$
Then $D_n$ is valid in every type-1 representable lattice but not in
$A_n$.

``...

``Note that for $n=3$, $D_n$ is equivalent to the Arguesian
law \cite{Haim1985b}, so $A_3(K)$ is a non-Arguesian lattice.  These
are in fact minimal non-Arguesian lattices which were first discovered
by Pickering [9].''

\cite{Haim1991}*{sec.~3 Th.~2} ``Theorem 2.  Every proper sublattice
of $A_n$ has a type-1 representation, and even a linear representation
over $K$.''

\cite{Haim1991}*{sec.~3~Cor.~5}
``Corollary 5. No finite set of lattice identities, or even universal Horn
sentences, i.e. sentences of the form
$\forall x_1 \cdots \forall x_k
[P_1(x) = Q_1(x)] \wedge \cdots \wedge [P_k(x) = Q_k(x)]
\Rightarrow P_0(x) = Q_0(x)$,
can characterize lattices with type-1
representation.''

\paragraph{Hawrylycz,
\textnormal{``Arguesian identities in invariant theory''}} \cite{Hawr1996}

Uses ``Arguesian'' and ``linear lattice''.

\cite{Hawr1996}*{sec.~4 eq.~14} ``It can be shown \cite{Haim1985a} that the
Arguesian law may be written
$$ c \wedge ([a \vee a^\prime) \wedge (b \vee b^\prime)] \vee c^\prime)
\leq a \vee ([((a \vee b) \wedge (a^\prime \vee b^\prime)) \vee ((b \vee c) \wedge (b^\prime \vee c^\prime))] \wedge (a^\prime \vee c^\prime))
\text{ ''} $$

\cite{Hawr1996}*{sec.~4 eqs.~15 and 16}\label{eq:higher} ``The
$N$th higher Arguesian law as given by Haiman \cite{Haim1985a} may be
written, given alphabets of letters $a_1, a_2 , \ldots, a_n$, and
$b_1, b_2, \ldots, b_n$ as
$$ a_n \wedge \left( \left[ \bigwedge_{i=1}^{n-1} (a_i \vee b_i) \right] \vee b_n \right)
\leq a_1 \vee \left( \left[ \bigvee_{i=1}^{n-1} ((a_i \vee a_{i+1}) \wedge (b_i \vee b_{i+1})) \right] \wedge (b_1 \vee b_n) \right) $$

``...

``By applying Proposition 4.12 the $N$th higher Arguesian law may be written
in the following self-dual form.

``$N$th Higher Order Arguesian Law. Let $a_1, \ldots, a_n$,
$a_1^\prime , \ldots, a_n^\prime$ and $b_1, \ldots, b_n, b_1^\prime, \ldots,
b_n^\prime$ be alphabets. Then the following identity holds as a linear
lattice identity:

$$ a_n^\prime \wedge \left( \left[ \bigwedge_{i=1}^{n-1} ((a_i \wedge a_1^\prime) \vee (b_i \wedge b_i^\prime)) \right] \vee (b_n \wedge b_n^\prime) \right)
\leq a_1 \vee \left( \left[ \bigvee_{i=1}^{n-1} ((a_1^\prime \vee a_{i+1}) \wedge (b_i^\prime \vee b_{i+1})) \right] \wedge (b_1 \vee b_n^\prime) \right)
\text{ ''} $$

\paragraph{Herrmann,
\textnormal{``A review of some of Bjarni J\'onsson's results on
representation of arguesian lattices''}} \cite{Herr2013}

Uses ``arguesian''.

\cite{Herr2013}*{sec.~3} ``J\'onsson considered this construction of a lattice
$L$ for $L_i =
L(V_{iD_i})$ and $\dim V_{iD_i} = 3$. In \cite{Jons1954a}*{Th.~3.6}, he
chose $\dim v = 1$ and $D_1$, $D_2$ of distinct prime characteristic
to obtain an example of a simple lattice of height 5 isomorphic to a
lattice of permuting equivalences but not embeddable into the normal
subgroup lattice of any group.''

\paragraph{J\'onsson,
\textnormal{``On the representation of lattices''}} \cite{Jons1953b}

Uses ``Desarguesian'' but only applies it to projective planes.
Uses ``representation of type 1 (or 2 or 3)''.

\cite{Jons1953b}*{Intro.}
``It is therefore natural to ask whether every modular lattice is
isomorphic to a lattice of normal subgroups of some group.  As will be
shown in Section 2 below, the answer to this question is negative.''

\cite{Jons1953b}*{sec.~1}
``Definition 1.1.  By a representation of a lattice $A$ we mean an
ordered pair $\langle F, U \rangle $ such that $U$ is a set and $F$ is a function
which maps $A$ isomorphically onto a sublattice of the lattice of all
equivalence relations over $U$. We say that $\langle F,U \rangle $ is
\begin{enumerate}
\item[(i)] of type 1 if $F(x)+F(y) = F(x);F(y)$ for $x,y \in A$,
\item[(ii)] of type 2 if $F(x)+F(y) = F(x);F(y);F(x)$ for $x,y \in A$,
\item[(iii)] of type 3 if $F(x)+F(y) = F(x);F(y);F(x);F(y)$ for
$x,y \in A$.
\end{enumerate}
[where ``;'' is simple composition of relations] 
The equation in (i) is equivalent to the condition that
$F(x);F(y) = F(y);F(x)$.''

\cite{Jons1953b}*{sec.~1}
``It is well known that the lattice of all normal subgroups of a group
$G$ has a representation of type 1; we let $U = G$, and for each
normal subgroup $H$ of $G$ let $F(H)$ be the set of all ordered pairs
$\langle u, v \rangle $ with $u v^{-1} \in H$.  It is not known whether the converse
of this result holds.''

\cite{Jons1953b}*{sec.~1 Th.~1.2}
``Theorem 1.2.  If a lattice $A$ has a representation of type 2, then
$A$ is modular.''
This implies that every lattice with a representation of type 1 is
modular.

\cite{Jons1953b}*{sec.~2 Lem.~2.1}
``Lemma 2.1.  Every modular lattice $A$ which has a representation of
type 1 satisfies the following condition:  If $a_0, a_1, a_2, b_0,
b_1, b_2 \in A$,  and if
\begin{align*}
x & = (a_0+b_0)\cdot(a_1+b_1)\cdot(a_2+b_2), \\
y & = (a_0+a_1)\cdot(b_0+b_1)\cdot
[(a_0+a_2)\cdot(b_0+b_2)+(a_1+a_2)\cdot(b_1+b_2)], \\
\intertext{then}
x & \leq a_0\cdot(a_1+y)+b_0\cdot(b_1+y) \text{. ''}
\end{align*}

Note that the above gives J\'onsson's definition of the Arguesian condition,
but he does not label it as such.

\cite{Jons1953b}*{sec.~2 Th.~2.2}
``Theorem 2.2.  If $A$ is the lattice of all subspaces of a projective
plane $P$, then the following conditions are equivalent:  (1) $A$ has
a representation of type 1, (ii) $A$ satisfies the condition [of
Lemma~2.1], (iii) $P$ is Desarguesian, (iv) $A$ is isomorphic to a
lattice of subgroups of some Abelian group $G$.''

\cite{Jons1953b}*{sec.~2 Th.~2.3}
``Theorem 2.3.  A free modular lattice with four or more generators
does not have a representation of type 1.''

\cite{Jons1953b}*{sec.~3 Th.~3.7}
``Theorem 3.7.  Every modular lattice has a representation of type 2.''

\cite{Jons1953b}*{sec.~4 Th.~4.3}
``Theorem 4.3.  Every lattice has a representation of type 3.''

\cite{Jons1953b}*{sec.~5}
``In connection with the above results it is natural to ask whether
the class of all lattices which are isomorphic to lattices of
commuting equivalence relations can be characterized by means of
identities.  ...  Similar questions can be raised concerning lattices
which are isomorphic to lattices of normal subgroups of arbitrary
groups or to lattices of subgroups of Abelian groups.  In particular,
it would be interesting to know whether these three classes of
lattices are actually distinct''

\paragraph{J\'onsson,
\textnormal{``Modular lattices and Desargues' theorem''}} \cite{Jons1954a}

Uses ``Arguesian'' and ``lattice of commuting equivalence relations''.

\cite{Jons1954a}*{sec.~1 Def.~1.1}
``Definition 1.1.  A lattice $A$ is said to be projective if it is
complete, atomistic, complemented and modular, and satisfies the
following condition:  If $p$ is an atom of $A$, if $I$ is any set, and
if the elements $a_i\in A$ with $i \in I$ are such that
$p \leq \sum_{i\in I} a_i$, then there exists a finite subset $J$ of
$I$ such that $p \leq \sum_{i \in J} a_i$.''

\cite{Jons1954a}*{sec.~1 Th.~1.7}
``Theorem 1.7.  If $A$ is a projective lattice, then the following
conditions are equivalent:
\begin{enumerate}
\item[(i)] $A$ is isomorphic to a lattice of commuting equivalence
relations.
\item[(ii)] $A$ is isomorphic to a lattice of normal subgroups of a
group.
\item[(iii)] $A$ is isomorphic to a lattice of subgroups of an Abelian
group.
\item[(iv)] $A$ is isomorphic to the lattice of all subspaces of an
Arguesian projective space.
\item[(v)] For any $a, b \in A^3$, if
$$ y = (a_0+a_1)\cdot(b_0+b_1)\cdot[(a_0+a_2)\cdot(b_0+b_2) +
(a_1+a_2)\cdot(b_1+b_2)], $$
then
$$ (a_0+b_0)\cdot(a_1+b_1)\cdot(a_2+b_2) \leq a_0\cdot(y+a_1) +
b_0\cdot(y+b_1) \text{. ''} $$
\end{enumerate}

\cite{Jons1954a}*{sec.~1 Def.~1.8}
``Definition 1.8.  A lattice $B$ is said to be Arguesian if it
satisfies the following condition:  For every $a, b \in B^3$, if
$$ y = (a_0+a_1)\cdot(b_0+b_1)\cdot[(a_0+a_2)\cdot(b_0+b_2) +
(a_1+a_2)\cdot(b_1+b_2)], $$
then
$$ (a_0+b_0)\cdot(a_1+b_1)\cdot(a_2+b_2) \leq a_0\cdot(y+a_1) +
b_0\cdot(y+b_1) \text{. ''} $$

\cite{Jons1954a}*{sec.~1 Th.~1.9}
``Theorem 1.9.  Every arguesian lattice is modular.''

\cite{Jons1954a}*{sec.~1 Th.~2.14}
``Theorem 2.14.  If $B$ is a complemented modular lattice, then the
following conditions are equivalent:
\begin{enumerate}
\item[(i)] $B$ is Arguesian.
\item[(ii)] $B$ is isomorphic to a lattice of commuting equivalence
relations.
\item[(iii)] $B$ is isomorphic to a lattice of normal subgroups of a
group.
\item[(iv)] $B$ is isomorphic to a lattice of subgroups of an Abelian
group.
\item[(v)] $B$ is isomorphic to a lattice of subspaces of an
Arguesian projective space.''
\end{enumerate}

\cite{Jons1954a}*{sec.~3}
Without assuming complementation, each condition in Th.~2.14 implies
all of the preceding conditions.

\cite{Jons1954a}*{sec.~3 Th.~3.6}
``Theorem 3.6.  There exists a five dimensional modular lattice $B$
which is isomorphic to a lattice of commuting equivalence relations, but
not to a lattice of normal subgroups of a group.''

\cite{Jons1954a}*{sec.~3}
``Thus we see that even for finite modular lattices the conditions
(ii) and (iii) of theorem 2.14 are not equivalent.''

\paragraph{J\'onsson,
\textnormal{``Representations of lattices. II. Preliminary report.''}}
\cite{Jons1954c}

\cite{Jons1954c}
Within complemented modular lattices, representability by (i)
commuting equivalence relations, (ii) normal subgroups of a group,
(iii) subgroups of an Abelian group, (iv) subspaces of a (possibly
degenerate) Desarguesian projective space are equivalent.  Similarly
for modular lattices of dimension 4 or less.  But some 5-dimensional
modular lattices that can be represented by normal subgroups of a
group cannot be represented by commuting equivalence relations.

\paragraph{J\'onsson,
\textnormal{``Representation of modular lattices and of relation algebras''}}
\cite{Jons1959a}

\cite{Jons1959a}*{sec.~3}
``While there do exist modular lattices which are not isomorphic to
lattices of commuting equivalence relations, the only known examples
are of a more or less pathological nature, such as the lattice of all
subspaces of a non-Arguesian projective plane
(\cite{Jons1953b}).''

\cite{Jons1959a}*{sec.~3 Th.~2}
``Theorem 2. In order for a lattice $\mathfrak{A}= \langle A,
+, \cdot \rangle$ to be isomorphic
to a lattice of commuting equivalence relations it is necessary and
sufficient that the following condition be satisfied:

``($\Gamma^\prime$) Suppose $n$ is a positive integer, $a_0, a_1, \cdots, a_{2n},
z \in A$ and, for each
positive integer $k \leq n$, $\phi(k)$ and $\psi(k)$ are natural
numbers with $\phi(k) \leq k$ and $\phi(k) \leq k$. Let
$$ b_{1,0,1} = b_{1,1,0} = a_0,\ b_{1,0,0} = b_{1,1,1} = z, $$
and for $k = 1,2, \cdots, n$ and $i, j = 0, 1, \cdots , k$ let
\begin{align*}
b_{k+1,i,j} & = \prod_{p \leq k} (b_{k,i,p} + b_{k,p,j}), \\
b_{k+1,i,k+1} & = b_{k+1,k+1,i} = (b_{k+1,i,\phi(k)} + a_{2k-1})
\cdot (b_{k+1,i,\psi(k)} + a_{2k}), \\
b_{k+1.k+1,k+1} & = z.
\end{align*}
With these notations, if $b_{k,\phi(k),\psi(k)} \leq a_{2k-1} + a_{2k}$
for $k = 1, 2, \cdots, n$, then $a_0 \leq
b_{n+1, 0, 1}$.''

\cite{Jons1959a}*{sec.~4 Prob.~1}
``Problem 1. Can the infinite sets of axioms [Horn sentences]
contained in the conditions ($\Gamma$) and ($\Gamma^\prime$) [above]
be replaced by finite sets of axioms? By sets of equations
[identities]?''

\cite{Jons1959a}*{sec.~4}
``It is an open question whether, conversely, every Arguesian lattice
is isomorphic to a lattice of commuting equivalence relations, but we
know that this is the case for complemented lattices
(\cite{Jons1954a}*{Th.~2.14})''

\paragraph{J\'onsson,
\textnormal{``The class of Arguesian lattices is self-dual''}}
\cite{Jons1972}

Uses ``Arguesian''.

\cite{Jons1972}
``The class of Arguesian lattices is self-dual.''

\cite{Jons1972}
Shows that the Arguesian condition:
for all $a_0, a_1, a_2, b_0, b_1, b_2 \in L$,
$$
(a_0 + b_0) (a_1 + b_1) \leq a_2 + b_2 \Rightarrow 
(a_0 + a_1) (b_0 + b_1) \leq (a_0 + a_2) (b_0 + b_2) + (a_1 + a_2) (b_1 + b_2)
$$
implies the dual of the condition:
for all $x_0, x_1, x_2, y_0, y_1, y_2 \in L$
$$
x_0 y_0 + x_1 y_1 \geq x_2 y_2 \Rightarrow
x_0 x_1 + y_0 y_1 \geq (x_0 x_2 + y_0 y_2) (x_1 x_2 + y_1 y_2).
$$

\paragraph{Lampe,
\textnormal{``A perspective on algebraic representations of lattices''}}
\cite{Lamp1994}

Uses ``Arguesian'' and ``type-1''.

\cite{Lamp1994} ``Modular laws
$$ \textnormal{(a) } z \leq x \Rightarrow x \wedge (y \vee z) = (x \wedge y) \vee z $$
$$ \textnormal{(b) } x \leq y \vee z \Rightarrow x \leq y \vee [z \wedge (x \vee y)]
\text{ ''} $$

\cite{Lamp1994} ``A lattice $L$ is Arguesian iff it satisfies the
following identity
$$ \bigwedge_{i<3} (x_i \vee y_i) \leq (x_0 \wedge (x_1 \vee m)) \vee (y_0 \wedge (y_1 \vee m)) $$
where
$$ m = (x_0 \vee x_1) \wedge (y_0 \vee y_1) \wedge [\{(x_0 \vee x_2) \wedge (y_0 \vee y_2)\} \vee \{(x_2 \vee x_1) \wedge (y_2 \vee y_1)\}] \text{. ''} $$

\paragraph{Kiss and P\'alfy,
\textnormal{``A lattice of normal subgroups that is not embeddable
into the subgroup lattice of an Abelian group''}}
\cite{KissPal1998}

\cite{KissPal1998}*{sec.~1}
``Theorem. The lattice of normal subgroups of the three generator free
group $G$ in the group variety defined by the laws $x^4 - 1$ and
$x^2y - yx^2$ cannot
be embedded into the subgroup lattice of any abelian group.''

\paragraph{Nation,
\textnormal{``J\'onsson's contributions to lattice theory''}}
\cite{Nat1994}

Uses ``Arguesian''.

\cite{Nat1994}
``Bjarni defined a representation of a lattice by equivalence
relations to be of \emph{type $n$} if only $n$ relational compositions
are required to achieve the join. Thus, for example, a representation
by permuting [commuting] equivalence relations would be of type 1.
His first result improved Whitman's theorem to show that every lattice
has a type 3 representation. The next result is even better.  Theorem
(J\'onsson). $L$ has a type 2 representation if and only if it is
modular.''

\cite{Nat1994}
``Bjarni showed in that there is a simpler equation with the same
property, which is now known as the \emph{Arguesian equation}:
$$ (a_0 \vee b_0) \wedge (a_1 \vee b_1) \wedge (a_2 \vee b_2) \leq a_0 \vee
(b_0 \wedge (c \vee b_1)) $$
where
$$ c = c_2 \wedge (c_0 \vee c_1) $$
and
$$ c_i = (a_j \vee a_k) \wedge (b_j \vee b_k) $$
for $\{i,j,k\} = \{0, 1, 2\}$. [...]
Now the Arguesian equation is still pretty horrible as
equations go, but Bjarni showed that it really is equivalent to the
condition
$$ (a_0 \vee b_0) \wedge (a_1 \vee b_1) \leq a_2 \vee b_2
\TNimplies
c_2 \leq c_0 \vee c_1 $$
where $c_i$ is as above, which is the lattice form of Desargues' Law.''

\cite{Nat1994}
``every lattice having a type 1 representation is Arguesian.''

\cite{Nat1994}
``Haiman found examples of Arguesian lattices which don't have a type
1 representation. In fact, Haiman constructed a sequence of examples
which give a stronger result. The class of all lattices with a type 1
representation forms a quasivariety $\mathcal{Q}$, and Bjarni gave an infinite
set of Horn sentences determining $\mathcal{Q}$.

``Theorem (Haiman). The
quasivariety $\mathcal{Q}$ of all lattices with a type 1 representation is not
finitely based.

``So in particular, a single equation like the
Arguesian law does not characterize $\mathcal{Q}$. It is unknown
whether $\mathcal{Q}$ is
a variety, i.e., if it is closed under homomorphic images.''

\paragraph{Nation, \emph{Notes on Lattice Theory}} \cite{Nat2017}

Uses ``Arguesian''.  Uses ``lattice of permuting equivalence relations''
and ``representation of type 1 (or 2 or 3)''.

\cite{Nat2017}*{ch.~4}``A lattice is said to be Arguesian if it
satisfies
$$ \textnormal{($A$) } (a_0 \vee b_0) \wedge (a_1 \vee b_1) \leq a_2 \vee b_2 \TNimplies c_2 \leq c_0 \vee c_1 $$
where
$$ c_i = (a_j \vee a_k) \wedge (b_j \vee b_k) $$
for $\{i, j, k\} = \{0, 1, 2\}$. The Arguesian law is (less obviously)
equivalent to a lattice inclusion,
$$ \textnormal{($A^\prime$) } (a_0 \vee b_0) \wedge (a_1 \vee b_1) \wedge (a_2 \vee b_2) \leq a_0 \vee (b_0 \wedge (c \vee b_1)) $$
where
$$ c = c_2 \wedge (c_0 \vee c_1) \text{. ''} $$

\cite{Nat2017}*{ch.~4 Th.~4.6 Cor.}``Theorem 4.6. If $\mathcal{L}$ is a
sublattice of $\textbf{Eq } X$ with the
property that $R \vee S = R \circ S$ for all $R, S \in \mathcal{L}$,
then $\mathcal{L}$ satisfies the Arguesian law.

``Corollary. Every lattice that has a type 1 representation is Arguesian.''

\cite{Nat2017}*{ch.~4}``Note that the normal subgroup lattice of a
group $G$ has a natural representation $(X, F)$: take $X = G$ and
$F(N) = \{(x, y) \in G^2:  xy^{-1} \in N\}$. This representation is in
fact type 1 (Exercise 3).''

\cite{Nat2017}*{ch.~4}``It follows that the lattice of normal
subgroups of a group is not only modular, but Arguesian; see exercise
3. All these types of lattices are Arguesian [and have a type 1
representation]: the lattice of subgroups of an abelian group, the
lattice of ideals of a ring, the lattice of subspaces of a vector
space, the lattice of submodules of a module. More generally, Ralph
Freese and Bjarni J\'onsson proved that if $\mathcal{V}$ is a variety
of algebras, all of whose congruence lattices are modular (such as
groups or rings), then the congruence lattices of algebras in
$\mathcal{V}$ are Arguesian [4].''

\cite{Nat2017}*{ch.~4} ``Interestingly, P.\ P.\ P\'alfy
and Laszlo Szab\'o have shown that subgroup lattices of abelian groups
satisfy an equation that does not hold in all normal subgroup lattices
\cite{PalSzab1995}.''

\cite{Nat2017}*{ch.~4} ``The question remains: \emph{Does there exist a set
of equations $\Sigma$ such that a lattice has a type 1 representation
if and only if it satisfies all the equations of $\Sigma$?} Haiman
proved that if such a $\Sigma$ exists, it must contain infinitely many
equations. In Chapter 7 we will see that a class of lattices is
characterized by a set of equations if and only if it is closed with
respect to direct products, sublattices, and homomorphic images.  The
class of lattices having a type 1 representation is easily seen to be
closed under sublattices and direct products, so the question is
equivalent to: \emph{Is the class of all lattices having a type 1
representation closed under homomorphic images?}''

\paragraph{Nation,
\textnormal{``Tribute to Bjarni J{\'o}nsson''}}
\cite{Nat2018}

Uses ``Arguesian''.

\paragraph{Nation and Pickering,
\textnormal{``Arguesian lattices whose skeleton is a chain''}}
\cite{NatPick1987}

Uses ``arguesian'' and ``type-1-representable''.

\cite{NatPick1987}
\cite{Jons1953b} points out that if $L$ is linear it satisfies
``the \emph{arguesian} lattice identity,
$$ (a_0 + b_0)(a_1 + b_1)(a_2 + b_2) \leq a_0 + b_0(b_1 + c_2(c_0 + c_1)) $$
where
$$ c_i = (a_j + a_k)(b_j + b_k) $$
for $i \neq j \neq k \neq i$.''

\cite{NatPick1987}
``\cite{Haim1987} has shown that the converse is false:
there exist arguesian lattices which do not have a type 1 representation.
In fact.
no finite set of universal Horn sentences suffices to characterize the class of
lattices having a type 1 representation.''

\cite{NatPick1987}
``\cite{Jons1959a} has given an infinite set of
universal Horn sentences which does characterize [linear lattices].''

\cite{NatPick1987}
``the classical theorem of projective geometry (see [1], [4]): If $L$ is a
finite dimensional complemented arguesian lattice, then $L$ is a
direct product of
lattices $L_i$, each of which is isomorphic to the lattice of all
subspaces of a vector
space over a division ring $K_i$. This naturally induces a type 1
representation of $L$.''

\cite{NatPick1987}
``Theorem. If $L$ is a finite dimensional arguesian lattice and [the
skeleton of $L$] is a chain, then $L$ has a type 1 representation.''

\cite{NatPick1987}
``It is easy to see that a subdirect product of type-1-representable lattices is
type-1-representable \cite{Jons1959a}.''

\cite{NatPick1987}*{Proof of Lem.~2}
``for a finite dimensional modular lattice, [$L$ is simple] is
equivalent to every pair of prime quotients in $L$ being
projective.''

\cite{NatPick1987}*{Proof of Lem.~4}
``For a fixed lattice $L$, the type 1 representations of $L$ form an
elementary class of relational structures. Using this, it is not hard
to show that a
lattice $L$ has a type 1 representation if and only if every finitely generated
sublattice of $L$ has a type 1 representation \cite{Jons1959a}.
Therefore [if every countable sublattice of $L$ is
type-1-representable, $L$ is.]''

\cite{NatPick1987}*{Proof of Lem.~4}
``Moreover, using the construction from \cite{Jons1959a} (or a
simple direct argument), if a countable lattice $L$ has a type 1
representation, it has a type 1 representation with a countable base
set.''

\paragraph{Ore, \textnormal{``Theory of equivalence relations''}}
\cite{Ore1942}

\cite{Ore1942}*{Ch.~1 sec.~5 eq.~5}
``One of the most important relations in the theory
of structures is the so-called \emph{Dedekind law}
\begin{equation*}
\tag{4} A \cap (B \cup C) = B \cup (A \cap C) \qquad (A \supset B).
\end{equation*}
Sometimes it is also formulated
\begin{equation*}
\tag{5} A = B \cup (A \cap C) \qquad
(B \cup C \supset A \supset B) \textnormal{. ''}
\end{equation*}

\paragraph{P\'alfy,
\textnormal{``Groups and lattices''}} \cite{Pal2001}

Uses ``arguesian''.

\cite{Pal2001}*{sec.~2 Th.~2.1}
``Theorem 2.1 (Whitman, 1946) Every lattice is isomorphic to a sublattice of
the subgroup lattice of some group.''

\cite{Pal2001}*{sec.~2 Th.~2.2}
``Theorem 2.2 (\cite{PudTum1980}) Every finite lattice is isomorphic to
a sublattice of the subgroup lattice of some finite group.''

\cite{Pal2001}*{sec.~2}
``Recall that a lattice $L$ is called distributive if the following
equivalent conditions hold for every $x$, $y$, $z \in L$:
\begin{enumerate}
\item $x \vee (y \wedge z)=(x \vee y) \wedge (x \vee z)$;
\item $x \wedge (y \vee z)=(x \wedge y) \vee (x \wedge z)$;
\item $(x \vee y) \wedge (x \vee z) \wedge (y \vee z)=
(x \wedge y) \vee (x \wedge z) \vee (y \wedge z)$.''
\end{enumerate}

\cite{Pal2001}*{sec.~3 Th.~3.3}
``Theorem 3.3 (Gr\"atzer and Schmidt, 1963) For every algebraic
lattice $L$ there exists an algebra $\boldsymbol{A}$ such that
$\textnormal{Con } \boldsymbol{A} \cong L$ [the lattice of congruences
of $\boldsymbol{A}$ is isomorphic to $L$].''

\cite{Pal2001}*{sec.~3 Prob.~3.4}
It is unknown whether for every finite lattice $L$ there exists a
finite algebra $\boldsymbol{A}$ such that
$\textnormal{Con } \boldsymbol{A} \cong L$.

\cite{Pal2001}*{sec.~4}
Modular lattices can be defined via a number of equivalent conditions:
\begin{enumerate}
\item $x \geq z$ implies $(x \wedge y) \vee z = x \wedge (y \vee z)$,
\item $x \geq z$ implies $(x \wedge y) \vee z \geq x \wedge (y \vee z)$,
\item $(x \wedge y) \vee (x \wedge z) = x \wedge [y \vee (x \wedge z)]$,
\item the lattice does not contain the pentagon sublattice.
\end{enumerate}

\cite{Pal2001}*{sec.~4}
``There exist laws of normal subgroup lattices that are even stronger
than modularity. The most important one is the \emph{arguesian law}
introduced by Bjarni J\'onsson in 1954. (The idea appeared earlier
in \cite{Schut1945}.) This is a translation of
Desargues' Theorem from projective geometry into the language of
lattices. Among the several equivalent formulations we prefer the
following form:
$$ X_1 \wedge\{Y_1 \vee [(X_2 \vee Y_2)\wedge(X_3 \vee Y_3)]\} \leq
[(Q_{12} \vee Q_{23})\wedge(Y_1 \vee Y_3)]\vee X_3, $$
where $Q_{ij} = (X_i \vee X_j)\wedge(Y_i \vee Y_j)$.''

\cite{Pal2001}*{sec.~4}
``So the arguesian law holds in the subspace lattice of a projective
plane if and only if Desargues' Theorem is true in the geometry. Since
there are nonarguesian planes, the arguesian law is stronger than the
modular law, as the subspace lattice is always modular.''

\cite{Pal2001}*{sec.~4 Th.~4.1}
``Theorem 4.1 (J\'onsson, 1954) The arguesian law holds in the normal
subgroup lattice of every group.''

\cite{Pal2001}*{sec.~4}
``In fact the arguesian law holds in every lattice consisting of
[commuting] equivalence relations.''

\cite{Pal2001}*{sec.~4}
``Mark Haiman in 1987 discovered a sequence of laws, the higher arguesian
identities
$$ X_1 \wedge \left[ Y_1 \vee \bigwedge_{i=2}^n (X_i \vee Y_i) \right] \leq
\left[ \bigvee_{i=1}^{n-1} Q_{i,i+1} \wedge (Y_1 \vee Y_n) \right] \vee X_n, $$
where $Q_{ij} = (X_i \vee X_j) \wedge (Y_i \vee Y_j)$, each one being
strictly stronger than the
previous one, that all hold in every lattice consisting of permuting
equivalence relations. ... Like the modular and the arguesian laws, the higher
arguesian identities hold not only in subgroup lattices of abelian
groups (as suggested by the underlying geometry), but also in normal
subgroup lattices of arbitrary groups.''

\cite{Pal2001}*{sec.~4}
``Later it was proved by Ralph Freese \cite{Frees1994} that
there is no finite basis for the laws of the class of all normal
subgroup lattices.''

\cite{Pal2001}*{sec.~4}
``the higher arguesian identities hold not only in subgroup lattices
of abelian groups (as suggested by the underlying geometry), but also
in normal subgroup lattices of arbitrary groups.''

\cite{Pal2001}*{sec.~4}
``Problem 4.2 (\cite{Jons1954a}; \cite{Birk1967}*{p.~179}) Can one
embed the normal subgroup lattice of an arbitrary group into the
subgroup lattice of an abelian group? Do all the laws of subgroup
lattices of abelian groups hold in normal subgroup lattices?''

\paragraph{P\'alfy and Szab\'o,
\textnormal{``An identity for subgroup lattices of Abelian groups''}}
\cite{PalSzab1995}

Uses ``Arguesian''.

\cite{PalSzab1995}
``subgroup lattices of Abelian groups satisfy the modular identity''

\cite{PalSzab1995}
``\cite{Jons1953b} introduced a lattice identity
equivalent to Desargues' theorem in projective geometry:
$$ (x_0 \vee y_0) \wedge (x_1 \vee y_1) \wedge (x_2 \vee y_2) \leq
[(z \vee x_1) \wedge x_0] \vee [(z \vee y_1) \wedge y_0], $$
where $z = z_{01} \wedge (z_{02} \vee z_{12})$
with $z_{ij} = (x_i \vee x_j) \wedge (y_i \vee y_j)$.
\cite{Haim1985a} discovered some stronger ``higher Arguesian
identities''. All these identities hold in subgroup lattices of
Abelian groups.  Moreover, they even hold in every lattice consisting
of pairwise permuting [commuting] equivalence relations, in
particular, in the lattice of normal subgroups of an arbitrary group.''

\cite{PalSzab1995}
``Theorem. The identity
$$ x_1 \wedge \{y_1 \vee [(x_2 \vee y_2) \wedge (x_3 \vee y_3) \wedge (x_4 \vee y_4)]\}
\leq [(p_{12} \vee p_{34}) \wedge (p_{13} \vee p_{24}) \wedge (p_{14} \vee p_{23})] \vee x_2 \vee y_3 \vee y_4 $$
where $p_{ij} = (x_i \vee y_j) \wedge (x_j \vee y_i)$, holds in the
subgroup lattice of every Abelian group but fails in the lattice of
normal subgroups of some finite group.''

\paragraph{Pudl\'ak and T\r{u}ma,
\textnormal{``Every finite lattice can be embedded in a finite partition lattice''}}
\cite{PudTum1980}

\cite{PudTum1980}*{final Th.}
``Theorem. For every [finite] lattice $L$, there exists a positive
integer $n_0$, such that for
every $n \geq n_0$, there is a normal embedding [preserving $\hat{0}$
and $\hat{1}$] $\phi:L \rightarrow \textnormal{Eq }(A)$, where $|A| = n$.''

\paragraph{Whitman,
\textnormal{``Lattices, equivalence relations, and subgroups''}}
\cite{Whit1946}

\cite{Whit1946}*{sec.~5 Th.~1}
``Theorem 1.  Any lattice is isomorphic to a sublattice of the lattice
of all equivalence relations on some set.''

\paragraph{Yan,
\textnormal{``Distributive laws for commuting equivalence relations''}}
\cite{Yan1998}

\cite{Yan1998}
``\cite{Ore1942} found necessary and sufficient conditions
under which the modular and distributive laws hold in the lattice of
equivalence relations on a set $S$. In the present paper, we consider
commuting equivalence relations. It has been proved by \cite{Jons1953b}
that the modular law holds in the lattice of commuting equivalence
relations. We give some necessary and sufficient conditions for the
distributive law and its dual to hold for commuting equivalence
relations.''

\section{Survey by properties} \label{sec:by-property}

\paragraph{Generalities}

We use ``Arguesian'' rather than ``arguesian'', as that is the
consensus of authors, and by analogy with, ``Boolean'', ``Euclidean'', and
``Darwinian''.  By contrast, ``abelian'' is often used by authors
instead of ``Abelian''.
We use ``linear'' for lattices that can be represented by commuting
equivalence relations, following \cite{Haim1984a}*{Intro.}, which follows
Gian-Carlo Rota.

\begin{definition}
An \emph{equation} or \emph{identity} within a type of abstract
algebra is a statement
$$ P(\boldsymbol{x}) = Q(\boldsymbol{x}) $$
where $\boldsymbol{x}$ is a (possibly empty) finite set of variables
that take values from the set of elements of the abstract algebra, and $P$
and $Q$ are formulas composed of the variables $\boldsymbol{x}$, the
constants of the type, and the operations of the type.
\end{definition}

\begin{definition}\cite{Gratz1979}*{App.~4 sec.~63}
A \emph{variety} or \emph{equationally defined} class of algebras is
the class of algebras of which all elements satisfy a particular (not
necessarily finite) set of equations.
\end{definition}

\begin{theorem}(Birkhoff's Variety Theorem \cite{Birk1935}*{sec.~10 Th.~10}\cite{Gratz1979}*{sec.~23 and sec.~26 Th.~3}\cite{WikiVar}) \label{th:birk-variety}
A class of algebras is a variety iff the class is closed under taking
subalgebras, direct products, and homomorphic images.
\end{theorem}

\begin{definition}
An \emph{implication}, \emph{equational implication},
\emph{universal Horn sentence}, or \emph{sentence} within a type
of abstract algebra is a statement
$$ P_1(\boldsymbol{x}) = Q_1(\boldsymbol{x}) \TNand
   P_2(\boldsymbol{x}) = Q_2(\boldsymbol{x}) \TNand
   \ldots
   P_n(\boldsymbol{x}) = Q_n(\boldsymbol{x})
   \TNimplies P(\boldsymbol{x}) = Q(\boldsymbol{x}) $$
where $\boldsymbol{x}$ is a (possibly empty) finite set of variables
that take values from set of elements of the abstract algebra,
$n \geq 0$, $P_i$, $Q_i$, $P$, and $Q$
are formulas composed of the variables $\boldsymbol{x}$, the
constants of the type of abstract algebra, and the
operations of the type of abstract algebra.
\end{definition}

\begin{definition}\cite{Gratz1979}*{App.~4 sec.~63}
An \emph{implicationally defined} class of algebras is
the class of algebras of which all elements satisfy a particular (not
necessarily finite) set of implications.
\end{definition}

\begin{theorem}\cite{Gratz1979}*{App.~4 sec.~63 Th.~3}
A class of algebras is implicationally defined iff the class is closed
under taking subalgebras, direct products, isomorphic images, and
direct limits.
\end{theorem}

\begin{theorem}
In a lattice, $P \leq Q$ iff $P \vee Q = Q$ iff $P \wedge Q = P$.
Thus, a statement $P \leq Q$ can be considered an abbreviation for either
of these equations.
\end{theorem}

\begin{theorem}
An equation is a special case of an implication (with $n = 0$).
Thus, any equationally defined class of algebras is an implicationally
defined class of algebras.
\end{theorem}

\begin{definition}
A class of latices is \emph{self-dual} if, for every lattice in the
class, its dual is also a member of the class, or equivalently,
if the \emph{dual class} composed of the duals of all members of the
class is the same class.
\end{definition}

\paragraph{General lattices}

\begin{theorem}
The class of latices is defined by finite set of identities and is self-dual.
\end{theorem}
\begin{proof}
\leavevmode

The axioms of lattices are a finite set of identities, and the set of
them is self-dual.
\end{proof}

\begin{theorem} \cite{Birk1935}*{sec.~21 Th.~22}
Every sublattice of subgroups of a group is isomorphic with a
sublattice of equivalence relations on a set, and conversely.
Every sublattice of subgroups of a finite group is isomorphic with a
sublattice of equivalence relations on a finite set, and conversely.
\end{theorem}

\begin{theorem} \cite{Whit1946}*{sec.~5 Th.~1}\cite{Pal2001}*{sec.~2 Th.~2.1}
Any lattice is isomorphic to a sublattice of the lattice
of all equivalence relations on some set and thus to
a sublattice of subgroups of a group.
\end{theorem}

\begin{theorem} \label{th:finite-lattice-finite-part}
\cite{PudTum1980}*{final Th.}\cite{Pal2001}*{sec.~2 Th.~2.2}
Every finite lattice is isomorphic to
a sublattice of the equivalence relations of a finite set and also to
a subgroup lattice of some finite group.
\end{theorem}

\begin{definition}
For any abstract algebra $\boldsymbol{A}$,
$\textnormal{Con } \boldsymbol{A}$ is the lattice of congruences
of $\boldsymbol{A}$.
\end{definition}

\begin{theorem} \cite{Pal2001}*{sec.~3 Th.~3.3}
For every algebraic
lattice $L$ there exists an algebra $\boldsymbol{A}$ such that
$\textnormal{Con } \boldsymbol{A}$ is isomorphic to $L$.
\end{theorem}

\begin{question} \label{q:finite-algebra} \cite{Pal2001}*{sec.~3 Prob.~3.4}
It is unknown whether for every finite lattice $L$ there exists a
finite algebra $\boldsymbol{A}$ such that
$\textnormal{Con } \boldsymbol{A} \cong L$.
Is this true?
\end{question}

\begin{definition}
For two equivalence relations $x, y$ on a set $U$, we say $x \leq y$ if
$x \subset y$ considering $x, y$ as sets of pairs of elements of $U$,
or equivalently if $a \equiv_x b$ implies $a \equiv_y b$ for all $a,
b \in U$.
Parallelly, for two partitions $x, y$ of a set $U$, we say $x \leq y$
if if $x$ is finer than $y$, that is, if every block of $x$ is
contained in a block of $y$.

Thus the minimum element in the lattice of equivalence relations of a
set is the discrete or
identity equivalence relation, and the maximum element is the
indiscreet or universal equivalence relation.  The meet of $x$ and $y$
is composed of the intersection of the equivalence classes of $x$ and
$y$, and the join of $x$ and $y$ is composed of the transitive closures of the
collective elements of $x$ and $y$.
\end{definition}

\begin{definition} \label{def:type} \cite{Jons1953b}*{sec.~1}
By a \emph{representation} of a lattice $A$ we mean an
ordered pair $\langle F, U \rangle$ such that $U$ is a set and $F$ is a function
which maps $A$ isomorphically onto a sublattice of the lattice of all
equivalence relations over $U$ or equivalently a sublattice of the
lattice of partitions of $U$. We say that $\langle F,U \rangle$ is
\begin{enumerate}
\item of type 1 if $F(x)+F(y) = F(x) \circ F(y)$ for $x,y \in A$,
\item of type 2 if $F(x)+F(y) = F(x) \circ F(y) \circ F(x)$ for $x,y \in A$,
\item of type 3 if $F(x)+F(y) = F(x) \circ F(y) \circ F(x) \circ F(y)$ for
$x,y \in A$.
\end{enumerate}
where $\circ$ is simple composition of relations.
\end{definition}

\begin{theorem} \label{th:all-type-3} \cite{Jons1953b}*{sec.~4 Th.~4.3}
Every lattice has a representation of type 3.
\end{theorem}

\begin{question} \label{q:finite-type-3}
Given theorems~\ref{th:finite-lattice-finite-part}
and~\ref{th:all-type-3}, is it true that every finite lattice has a
type-3 representation over a finite base set?
\end{question}

\paragraph{Modular}

\begin{theorem} \label{th:modular}
For a lattice $L$, the following conditions are equivalent:
\begin{enumerate}

\item \label{th:modular:i1} \cite{Day1982}*{sec.~1}\cite{Gratz1996}*{Ch.~I sec.~4 Lem.~12}\cite{Pal2001}*{sec.~4}\cite{EnMathMod}
$L$ satisfies
$$ (x \wedge y) \vee (x \wedge z) = x \wedge (y \vee (x \wedge z)) $$

\item \textnormal{[dual to item \ref{th:modular:i1}]} $L$ satisfies
$$ (x \vee y) \wedge (x \vee z) = x \vee (y \wedge (x \vee z)) $$

\item \label{th:modular:i3} \cite{Ded1900}*{sec.~2 par.~VIII
eq.~8}\cite{Birk1967}*{ch.~I sec.~7}\cite{Ore1942}*{Ch.~1 sec.~5 eq.~4}\linebreak[0]\cite{Day1982}*{sec.~1}\linebreak[0]\cite{Lamp1994}\linebreak[0]\cite{Gratz1996}*{Ch.~IV sec.~1 Th.~1}\linebreak[0]\cite{Pal2001}*{sec.~4}\linebreak[0]\cite{EnMathMod}
$L$ satisfies
$$ x \geq z \TNimplies (x \wedge y) \vee z = x \wedge (y \vee z) $$

\item \label{th:modular:i4} \cite{Pal2001}*{sec.~4} $L$ satisfies
$$ x \geq z \TNimplies (x \wedge y) \vee z \geq x \wedge (y \vee z) $$

\item \cite{Birk1967}*{ch.~II sec.~7 Exer.~2}\cite{Gratz1996}*{Ch.~IV sec.~1 Th.~1} $L$ satisfies
$$ x \wedge (y \vee z) = x \wedge ((y \wedge (x \vee z)) \vee z) $$

\item \cite{Gratz1996}*{Ch.~IV sec.~1 Th.~1} $L$ satisfies
$$ x \vee (y \wedge z) = x \vee ((y \vee (x \wedge z)) \wedge z) $$

\item \label{th:modular:i9} \cite{Lamp1994} $L$ satisfies
$$ x \leq y \vee z \TNimplies x \leq y \vee [z \wedge (x \vee y)] $$

\item \textnormal{[dual to item \ref{th:modular:i9}]} $L$ satisfies
$$ x \geq y \wedge z \TNimplies x \geq y \wedge [z \vee (x \wedge y)] $$

\item \label{th:modular:i9a} \cite{Gratz1996}*{Ch.~IV Exer.~1} $L$ satisfies
$$ (x \vee (y \wedge z)) \wedge (y \vee z) =
(x \wedge (y \vee z)) \vee (y \wedge z) $$

\item \label{th:modular:i12} \cite{Haim1985a}*{Exam.~2.3} $L$ satisfies
$$ (b \vee c) \wedge a \leq b \vee (c \wedge (a \vee b)) $$

\item \label{th:modular:i14} \cite{Haim1985a}*{Exam.~2.3} $L$ satisfies
$$ ((b_1 \wedge b_2) \vee c) \wedge a \leq
b_1 \vee (c \wedge (a \vee b_2)) $$

\item \textnormal{[dual to item \ref{th:modular:i14}]} $L$ satisfies
$$ ((b_1 \vee b_2) \wedge c) \vee a \geq
b_1 \wedge (c \vee (a \wedge b_2)) $$

\item \label{th:modular:i16} \cite{Birk1967}*{ch.~II sec.~7 Exer.~2} $L$ satisfies
$$ [(x \wedge z) \vee y] \wedge z = [(y \wedge z) \vee x] \wedge z $$

\item \textnormal{[dual to item \ref{th:modular:i16}]} $L$ satisfies
$$ [(x \vee z) \wedge y] \vee z = [(y \vee z) \wedge x] \vee z $$

\item \cite{Birk1967}*{ch.~II sec.~7 Lem.~1} $L$ satisfies
$$ x \geq y \TNand a \wedge x = a \wedge y \TNand
a \vee x = a \vee y \TNimplies x=y $$

\item \label{th:modular:i17} \cite{Ore1942}*{Ch.~1 sec.~5 eq.~5} $L$ satisfies
$$ b \vee c \geq a \geq b \TNimplies a = b \vee (a \wedge c) $$

\item \textnormal{[dual to item \ref{th:modular:i17}]} $L$ satisfies
$$ b \wedge c \leq a \leq b \TNimplies a = b \wedge (a \vee c) $$

\item \label{th:modular:i18} \cite{Birk1967}*{ch.~I sec.~7 Th.~12}\cite{Gratz1996}*{Ch.~IV sec.~1 Th.~1}\cite{Pal2001}*{sec.~4}
$L$ does not contain the lattice $N_5$ (the ``pentagon
  lattice'') as a sublattice.
\\
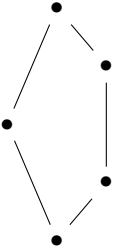
\begin{tikzcd}[column sep=0.5em,row sep=1em,every arrow/.append style={dash}]
  & \bullet \ar[ddl] \ar[dr] \\
  && \bullet \ar[dd] \\
   \bullet \ar[ddr] \\
  && \bullet \ar[dl] \\
  & \bullet
\end{tikzcd} \\
\end{enumerate}
\end{theorem}

\begin{definition}
A lattice satisfying any of the conditions of theorem~\ref{th:modular} is
called \emph{modular}.\footnote{\cite{Ore1942} uses \emph{Dedekind
  condition, law,} and \emph{relation} for ``modular'' and
  \emph{structure} for ``lattice''.}
\end{definition}

\begin{theorem}
The class of modular latices is defined by a finite set of identities
and is self-dual.  There are finite lattices that are not modular.
\end{theorem}
\begin{proof}
\leavevmode

The set of axioms of lattices, together with any one of the items of
theorem~\ref{th:modular} that is an identity, is a finite set of
identities characterizing modular lattices.
Choosing item~\ref{th:modular:i9a} as the modular identity makes the
set self-dual.
The lattice described in item~\ref{th:modular:i18}, $N_5$, is a finite
lattice that is not modular.
\end{proof}

\begin{theorem}
If $L$ is a finitary lattice, the following conditions are equivalent:
\begin{enumerate}
\item $L$ is modular
\item \cite{Birk1967}*{ch.~II sec.~8 Th.~16}
If $a \neq b$ both cover $c$, then there exists $d \in L$ which covers
both $a$ and $b$, and dually, 
if $a \neq b$ both are covered by $c$, then there exists $d \in L$
which is covered by both $a$ and $b$.
\item \cite{Birk1967}*{ch.~II sec.~8 Ths.~15 and 16}\cite{EnMathMod}
$L$ is graded with rank function $\rho$\footnote{Thus
satisfying the the Jordan--Dedekind chain condition: all maximal
chains between the same endpoints have the same finite length.}
and $\rho(x) + \rho(y) = \rho(x \vee y) + \rho(x \wedge y)$
\item
$x$ and $y$ both cover $x \wedge y$ iff $x$ and $y$ are both
covered by $x \vee y$
\end{enumerate}
\end{theorem}

\begin{theorem} \label{th:modular-type-2}
\cite{Jons1953b}*{sec.~3 Th.~3.7}\cite{Jons1953b}*{sec.~1 Th.~1.2}
A lattice is modular iff it has a representation of type 2.
\end{theorem}

\begin{question} \label{q:finite-type-2}
Given theorems~\ref{th:finite-lattice-finite-part}
and~\ref{th:modular-type-2}, is it true that every finite modular lattice has a
type-2 representation over a finite base set?
\end{question}

\begin{theorem} \cite{NatPick1987}*{Proof of Lem.~2}
If $L$ is a finite dimensional modular lattice, $L$ is simple iff
all pairs of prime quotients\footnote{That is, all covering pairs
$x \lessdot y$.} in $L$ are projective to each other.
\end{theorem}

\begin{theorem} \label{th:mod-proj} \cite{Pal2001}*{sec.~4}
The subspace lattice of a projective space (of any dimension, whether
or not it satisfies Desargues theorem) is modular.
\end{theorem}

\paragraph{Arguesian}

\begin{theorem} \label{th:arguesian}
For a lattice $L$, the following conditions are equivalent:
\begin{enumerate}

\item \cite{Jons1953b}*{sec.~2 Lem.~2.1}\cite{Jons1954a}*{sec.~1 Def.~1.8}\cite{DayPick1984}\linebreak[0]\cite{Lamp1994}\linebreak[0]\cite{Gratz1996}*{Ch.~IV sec.~4 Def.~9}
For $a_0, a_1, a_2, b_0, b_1, b_2 \in L$
$$ (a_0 \vee b_0) \wedge (a_1 \vee b_1) \wedge (a_2 \vee b_2) \leq
(a_0 \wedge (y \vee a_1)) \vee (b_0 \wedge (y \vee b_1)) $$
where
$$ y = (a_0 \vee a_1) \wedge(b_0 \vee b_1) \wedge
[((a_0 \vee a_2) \wedge (b_0 \vee b_2)) \vee ((a_1 \vee a_2) \wedge (b_1 \vee b_2))] $$

\item \cite{Haim1985a}*{Exam.~2.2}\cite{Haim1985b}*{eq.~1}\cite{Hawr1996}*{sec.~4 eq.~14}\linebreak[0]\cite{Pal2001}*{sec.~4}
For $X_1, X_2, X_3, Y_1, Y_2, Y_3 \in L$
$$ X_1 \wedge\{Y_1 \vee [(X_2 \vee Y_2)\wedge(X_3 \vee Y_3)]\} \leq
[(Q_{12} \vee Q_{23})\wedge(Y_1 \vee Y_3)]\vee X_3 $$
where $Q_{ij} = (X_i \vee X_j)\wedge(Y_i \vee Y_j)$

\item \label{th:arguesian:i3} \cite{Jons1953b}\cite{Jons1954a}\cite{Day1982}*{sec.~3 Def.~1}\cite{DayPick1984}\linebreak[0]\cite{Haim1985a}*{Exam.~1.1}\linebreak[0]\cite{Jons1972}\linebreak[0]\cite{Nat2017}*{ch.~4}\linebreak[0]\cite{EnMathArg}
For $a_0, a_1, a_2, b_0, b_1, b_2 \in L$
$$ (a_0 \vee b_0) \wedge (a_1 \vee b_1) \leq a_2 \vee b_2 \TNimplies c_2 \leq c_0 \vee c_1 $$
where
$$ c_i = (a_j \vee a_k) \wedge (b_j \vee b_k)
\TNfor \{i, j, k\} = \{0, 1, 2\} $$

\item \cite{Day1982}*{sec.~3 Th.~8}\cite{DayPick1984}\cite{Nat2017}*{ch.~4}
For $a_0, a_1, a_2, b_0, b_1, b_2 \in L$
$$ (a_0 \vee b_0) \wedge (a_1 \vee b_1) \wedge (a_2 \vee b_2) \leq a_0 \vee (b_0 \wedge (c \vee b_1)) $$
where
\begin{align*}
c_i & = (a_j \vee a_k) \wedge (b_j \vee b_k)
\TNfor \{i, j, k\} = \{0, 1, 2\} \\
c & = c_2 \wedge (c_0 \vee c_1)
\end{align*}

\item \label{th:arguesian:i5} \cite{Haim1985a}*{Exam.~2.2}
For $a_1, a_1^\prime, a_2, a_2^\prime, b_1, b_1^\prime, b_2, b_2^\prime, 
c, c_1^\prime, c_2^\prime \in L$
\begin{multline*}
c \wedge ([((a_1 \wedge a_2) \vee (a_1^\prime \wedge a_2^\prime))
\vee ((b_1 \wedge b_2) \vee (b_1^\prime \wedge b_2^\prime))]
\vee (c_1^\prime \wedge c_2^\prime)) \\
\leq a_1 \vee ([((a_2 \vee b_1) \wedge (a_2^\prime \vee b_1^\prime))
\vee ((b_2 \vee c) \wedge (b_2^\prime \vee c_1^\prime))] \wedge
(a_1 \vee c_2^\prime))
\end{multline*}
Note that the above is its own dual after exchanging the variables
$a_1 \leftrightarrow c$, $a_1^\prime \leftrightarrow c_1^\prime$,
$a_2 \leftrightarrow b_2$, and $a_2^\prime \leftrightarrow b_2^\prime$.

\item \cite{Day1982}*{sec.~3 Th.~7}\cite{DayPick1984}
For $x_0, x_1, x_2, y_0, y_1, y_2 \in L$
$$ (x_0 \vee x_1) \wedge (z \vee y_1) \leq
[(x_0 \vee x_2) \wedge (y_0 \vee y_2)] \vee
[(x_1 \vee x_2) \wedge (y_1 \vee y_2)] \vee [y_1 \wedge (x_0 \vee x_1)] $$
where
$$ z = y_0 \wedge [x_0 \vee ((x_1 \vee y_1) \wedge (x_2 \vee
y_2))] $$

\item \cite{Day1982}*{sec.~3 Th.~8}
For $a_0, a_1, a_2, b_0, b_1, b_2 \in L$
$$ (a_0 \vee b_0) \wedge (a_1 \vee b_1) \wedge (a_2 \vee b_2) \leq
a_0 \vee b_1 \vee \bar{c} $$
where
\begin{align*}
c_i & = (a_j \vee a_k) \wedge (b_j \vee b_k) \qquad \{i, j, k\} = \{0, 1, 2\} \\
\bar{c} & = c_2 \wedge (c_0 \vee c_1)
\end{align*}

\item \cite{Day1982}*{sec.~3 Th.~8}\cite{EnMathArg}
For $a_0, a_1, a_2, b_0, b_1, b_2 \in L$
$$ (a_0 \vee b_0) \wedge (a_1 \vee b_1) \wedge (a_2 \vee b_2) \leq
[a_0 \wedge (a_1 \vee \bar{c})] \vee [b_0 \wedge (b_1 \vee \bar{c})] $$
where
\begin{align*}
c_i & = (a_j \vee a_k) \wedge (b_j \vee b_k) \qquad \{i, j, k\} = \{0, 1, 2\} \\
\bar{c} & = c_2 \wedge (c_0 \vee c_1)
\end{align*}

\end{enumerate}
\end{theorem}

\begin{definition}
A lattice satisfying any of the conditions of
theorem~\ref{th:arguesian} is called \emph{Arguesian}.\footnote{This
concept appeared in \cite{Schut1945}.}
\end{definition}

\begin{theorem} \cite{Jons1954a}*{sec.~1 Th.~1.9}\cite{Day1982}*{sec.~3 Lem.~2}
Any Arguesian lattice is modular.
\end{theorem}

\begin{theorem} \cite{Jons1972}
The class of Arguesian lattices is defined by a finite set of identities and is
self-dual.
\end{theorem}
\begin{proof}
\leavevmode

\cite{Jons1972} proves this by
showing theorem~\ref{th:arguesian} item~\ref{th:arguesian:i3} implies
its own dual.
\cite{Haim1985a} proves this by the fact
that theorem~\ref{th:arguesian} item~\ref{th:arguesian:i5} is self-dual.
\end{proof}

\begin{theorem} \label{th:proj-arg} \cite{Pal2001}*{sec.~4}
The subspace lattice of a projective plane is Arguesian if and only if
Desargues' Theorem is true in the projective plane.
\end{theorem}

\begin{theorem}
There are finite modular lattices that are not Arguesian.
\end{theorem}
\begin{proof}
\leavevmode

Consider a finite projective plane in which Desargues' Theorem is not
true.
By theorem~\ref{th:mod-proj}, the subspace lattice of that projective
plane is modular, but by theorem~\ref{th:proj-arg} it is not Arguesian.
\end{proof}

\begin{theorem} \cite{Haim1985b}*{Note~2}
Any lattice identity equivalent to the Arguesian law must
necessarily involve at least six variables.
\end{theorem}
\begin{proof}
\leavevmode

\cite{Haim1985b}*{Note~2} constructs a non-Arguesian lattice all
of whose sublattices generated by five elements are Arguesian (and in
fact, are sublattices of a quaternionic projective geometry).
Thus, any identity with five or fewer variables that is true in
all Arguesian lattices is true in this lattice, and so the identity does
not exclude all non-Arguesian lattices.
\end{proof}

\begin{remark}
Reformulating the implication in theorem~\ref{th:arguesian}
item~\ref{th:arguesian:i3} into two equalities produces a form which 
strongly resembles the geometric statement
of Desargues Theorem in projective geometry:
For $a_0, a_1, a_2, b_0, b_1, b_2 \in L$
\begin{equation} \label{eq:arg-eq}
(a_0 \vee b_0) \wedge (a_1 \vee b_1) = a_2 \vee b_2 \TNimplies c_2 = c_0 \vee c_1
\end{equation}
where
$$ c_i = (a_j \vee a_k) \wedge (b_j \vee b_k)
\TNfor \{i, j, k\} = \{0, 1, 2\} $$
However, equation \ref{eq:arg-eq} is not straightforwardly equivalent to
item~\ref{th:arguesian:i3}, as in the Young-Fibonacci lattice of
degree 1, the instance
$ a_0 = 2, a_1 = 2, a_2 = 11, b_0 = 121, b_1 = 1211, b_2 = 21 $
satisfies the antecedent of \ref{eq:arg-eq} but not the consequent.
(However, it is unknown whether the Young-Fibonacci lattice is Arguesian.)
\end{remark}

\paragraph{Higher-order Arguesian}

\begin{theorem} \label{th:higher}
Given $n \geq 1$,
for all lattices $L$, the following conditions are equivalent:
\begin{enumerate}

\item  \label{th:higher:i1}
\cite{Haim1985a}*{sec.~2.0 Exam.~2.5, Figs.~6 and 7}\cite{Haim1987}*{sec.~3}\linebreak[0]\cite{Haim1991}*{sec.~3 Th.~1}\linebreak[0]\cite{Hawr1996}*{sec.~4 eq.~15}\linebreak[0]\cite{Pal2001}*{sec.~4}
For $a_1, a_2, \ldots, a_n \TNand b_1, b_2, \ldots, b_n \in L$
$$ a_n \wedge \left( \left[ \bigwedge_{i=1}^{n-1} (a_i \vee b_i) \right] \vee b_n \right)
\leq a_1 \vee \left( \left[ \bigvee_{i=1}^{n-1} ((a_i \vee a_{i+1}) \wedge (b_i \vee b_{i+1})) \right] \wedge (b_1 \vee b_n) \right) $$

\item \label{th:higher:i2} \cite{Hawr1996}*{sec.~4 eq.~16}
For $a_1, \ldots, a_n, a_1^\prime , \ldots, a_n^\prime, b_1, \ldots,
b_n, \TNand b_1^\prime, \ldots, b_n^\prime \in L$
$$ a_n^\prime \wedge \left( \left[ \bigwedge_{i=1}^{n-1} ((a_i \wedge a_1^\prime) \vee (b_i \wedge b_i^\prime)) \right] \vee (b_n \wedge b_n^\prime) \right)
\leq a_1 \vee \left( \left[ \bigvee_{i=1}^{n-1} ((a_1^\prime \vee a_{i+1}) \wedge (b_i^\prime \vee b_{i+1})) \right] \wedge (b_1 \vee b_n^\prime) \right) $$
Note that the above is self-dual under a suitable change of variables.

\end{enumerate}
\end{theorem}

\begin{definition} \label{def:higher}
\cite{Haim1991}*{sec.~3 Th.~1}\cite{Haim1991}*{sec.~3 Th.~1}
For any $n \geq 1$, the equivalent conditions in
theorem~\ref{th:higher} are called $D_n$ and the \emph{$n$-th order
Arguesian law}.  Collectively, the $D_n$ for $n \geq 4$ are called
\emph{higher-order Arguesian laws}.
\end{definition}

\begin{definition}
A lattice which satisfies $D_n$ is called \emph{$n$-th order Arguesian}.
\end{definition}

\begin{theorem}
$D_1$ is trivially true for any lattice.
$D_2$ is equivalent to the modular law.
$D_3$ is equivalent to the Arguesian law.
\end{theorem}

\begin{theorem}
For any $n \geq 1$, the class of $n$-th order Arguesian lattices is
defined by a finite set of identities and is self-dual.
\end{theorem}
\begin{proof}
\leavevmode

The defining identities are the axioms of lattices and $D_n$
[\ref{th:higher:i2}].  These identities are self-dual.
\end{proof}

\begin{theorem} \label{th:Dn-weakly-stronger}
For any $n \ge 1$, $D_{n+1}$ implies $D_n$.
\end{theorem}
\begin{proof}
\leavevmode

The theorem is true for $n = 1$ because $D_1$ is trivially true.

Alternatively, $n \geq 2$.
We start with this formulation of $D_{n+1}$, set $a_1 = a_2$ and
$b_1 = b_2$ and will show that the result is $D_n$.
\begin{equation}
a_{n+1} \wedge \left( \left[ \bigwedge_{i=1}^n (a_i \vee b_i) \right] \vee b_{n+1} \right)
\leq a_1 \vee \left( \left[ \bigvee_{i=1}^n ((a_i \vee
a_{i+1}) \wedge (b_i \vee b_{i+1})) \right] \wedge (b_1 \vee b_{n+1}) \right)
\label{eq:Dn}
\end{equation}
Consider the iterated meet on the left of equation~\ref{eq:Dn}.  The
substitution makes the first two terms equal, so
\begin{equation} \label{eq:Dn-meet}
\bigwedge_{i=1}^n (a_i \vee b_i) 
= \bigwedge_{i=2}^n (a_i \vee b_i)
\end{equation}
Consider the iterated join on the right of equation~\ref{eq:Dn}.  The
first two terms are
$ (a_1 \vee a_2) \wedge (b_1 \vee b_2) $ and
$ (a_2 \vee a_3) \wedge (b_2 \vee b_3) $.
Substituting turns the first term into $ a_2 \wedge b_2 $, which is
$\leq$ the second term and thus drops out of the join.
\begin{equation} \label{eq:Dn-join}
\bigvee_{i=1}^n ((a_i \vee a_{i+1}) \wedge (b_i \vee b_{i+1}))
= \bigvee_{i=2}^n ((a_i \vee a_{i+1}) \wedge (b_i \vee b_{i+1}))
\end{equation}
Now looking at equation~\ref{eq:Dn} as $D_{n+1}$,
substituting $a_1 = a_2, b_1 = b_2$, and using
equations~\ref{eq:Dn-meet} and~\ref{eq:Dn-join}
\begin{align*}
a_{n+1} \wedge \left( \left[ \bigwedge_{i=2}^n (a_i \vee b_i) \right] \vee b_{n+1} \right)
& \leq a_2 \vee \left( \left[ \bigvee_{i=2}^n ((a_i \vee
a_{i+1}) \wedge (b_i \vee b_{i+1})) \right] \wedge (b_2 \vee b_{n+1}) \right) \\
\end{align*}
which, with a change of variables, is $D_n$.
\end{proof}

\begin{definition}
\cite{Haim1987}*{sec.~2} defines a series of finite lattices named $A_n$ for
$n \geq 3$.
\end{definition}

\begin{theorem} \label{th:Dn-misc}
\leavevmode

\begin{enumerate}
\item \cite{Haim1985a}*{sec.~2.0 Exam.~2.5} Every $D_n$ is satisfied by
every linear[\ref{def:linear}] lattice.
\item $D_n$ can be written with $2n$ variables (in the form of defining
equation \ref{th:higher:i1}).
\item \cite{Haim1987}*{sec.~3} $D_n$ fails in $A_n$, and hence $A_n$
is not linear.
\item \cite{Haim1987}*{sec.~3} Every proper sublattice of $A_n$ is linear,
and hence satisfies all $D_n$'s.
\item \cite{Haim1987}*{sec.~3} At least $n$ elements are required to
generate $A_n$.
\end{enumerate}
\end{theorem}

\begin{theorem} \label{th:vars}
Any statement containing $< n$ variables that is true in all
linear[\ref{def:linear}] lattices is true in $A_n$.
\end{theorem}
\begin{proof}
\leavevmode

Given such a statement and values of its variables, the truth of that
instance of the
statement depends only on the sublattice generated by the values of
the variables.  Since the number of variables is less than the number
needed to generate $A_n$, that sublattice is a proper sublattice, and
hence linear.  Thus, the statement is true for those values of the
variables.  Since this follows for any values of the variables, the
statement is satisfied by $A_n$.
\end{proof}

\begin{theorem} \label{th:Dn-strictly-stronger}
The sequence of properties $D_{2^n-1}$ for $n \geq 2$ (specifically $D_3,
D_7, D_{15}, \ldots$) are an infinite sequence of identities of
strictly increasing strength on the class of finite lattices.
The sequence begins with $D_3$, which is the Arguesian condition.
\end{theorem}
\begin{proof}
\leavevmode

Consider $D_i$ and $A_j$ with $j > 2i$.
$D_i$ has $2i$ variables and by theorems~\ref{th:Dn-misc} and~\ref{th:vars},
so $D_i$ is valid in $A_j$.

By theorem~\ref{th:Dn-weakly-stronger}, $D_{2^{n+1}-1}$ is at least as
strong as $D_{2^n-1}$.
But $D_{2^n-1}$ is valid in $A_{2^{n+1}-1}$ and $D_{2^{n+1}-1}$ fails
in it, showing that $D_{2^{n+1}-1}$ is strictly stronger than
$D_{2^n-1}$.
\end{proof}

\begin{question} \label{q:Dn-strictly-stronger} \cite{Haim1987}*{sec.~4}
It is expected that $D_{n+1}$ is strictly stronger than $D_n$ for all
$n \geq 1$.  Is this true?  Is this true in the class of finite lattices?
\end{question}

\paragraph{Unbounded-order Arguesian}

\begin{definition}
The property $D_\infty$ or the \emph{unbounded-order Arguesian law} is the
conjunction of all properties $D_n$ for $n \geq 3$
[\ref{def:higher}].\footnote{From \cite{Haim1985a}*{sec.~2.0
Exam.~2.5, Figs.~6 and 7} it is clear that there is no transfinite
version of the sentences $D$, and thus no grounds for considering the
index of $D$ to be either a cardinal or an ordinal.  So instead of using
the well-defined $\aleph_0$ or $\omega$, we use the indicative
$\infty$.}
A lattice which satisfies $D_\infty$ is called \emph{unbounded-order
Arguesian}.
\end{definition}

\begin{theorem}
The unbounded-order Arguesian law is strictly stronger than all of the
$D_n$ on the class of finite lattices.
\end{theorem}
\begin{proof}
\leavevmode

This is because the $D_n$ contain an endless sequence of strictly
increasingly strong statements on the class of finite
lattices.[theorem~\ref{th:Dn-strictly-stronger}]
\end{proof}

\begin{theorem}
The class of unbounded-order Arguesian lattices is
defined by identities and is self-dual.
\end{theorem}

\begin{theorem} \label{th:uoa-infinite} \cite{Haim1987}*{sec.~4}\cite{Haim1991}*{sec.~3 Cor.~5}
The class of unbounded-order Arguesian lattices cannot be defined by a
finite set of statements.
\end{theorem}
\begin{proof}
\leavevmode

Given any finite set of statements, there is an $N$ that is greater
than the number of variables in any of the statements.
By theorem~\ref{th:vars}, all of the statements are true in $A_N$.
But theorem~\ref{th:Dn-misc} shows that $A_N$ does not satisfy $D_N$,
and so the set of statements cannot characterize unbounded-order
Arguesian lattices.
\end{proof}

\paragraph{Linear}

\begin{quotation}
While there do exist modular lattices which are not isomorphic to
lattices of commuting equivalence relations, the only known examples
are of a more or less pathological nature, such as the lattice of all
subspaces of a non-Arguesian projective plane (\cite{Jons1953b}).
--- Bjarni J\'onsson \cite{Jons1959a}*{sec.~3}
\end{quotation}

\begin{definition} \label{def:linear}
A lattice is \emph{linear} if it has a type 1 representation
[\ref{def:type}] by a
sublattice of equivalence relations on a set (equivalently partitions of a
set).
\end{definition}

\begin{theorem} \cite{Jons1953b}*{sec.~1}
A representation of a lattice is type 1 iff
for every two elements $x, y$, their representative equivalence
relations $\phi(x), \phi(y)$ commute when composed as general binary
relations: $\phi(x) \circ \phi(y) = \phi(y) \circ \phi(x)$.
\end{theorem}

\begin{theorem} \cite{BritzMainPezz2001}*{Th.~1}
Let $R$ and $T$ be equivalence relations on a set $S$.
The following are equivalent:
\begin{enumerate}
\item $R$ and $T$ commute as binary relations;
\item The join of $R$ and $T$ within the lattice of equivalence
relations on $S$ is equal to their composition $R \circ T$;
\item $R \circ T$ is an equivalence relation.
\end{enumerate}
\end{theorem}

\begin{question} \label{q:finite-type-1} \cite{Wor2024a}
Given theorem~\ref{th:finite-lattice-finite-part},
is it true that every finite linear lattice has a
type-1 representation over a finite base set?
\end{question}

Obviously, an infinite lattice does not have a (faithful) linear
representation with a finite base set.  However, we can investigate:

\begin{definition} \cite{Wor2024a}
A linear representation $\phi$ of a lattice $L$ is \emph{finite-block} if for
every $x \in L$, $\phi(x)$ (looked at as a partition) consists of only
finite blocks.
\end{definition}

\begin{theorem}
If a finitary lattice $L$ has a finite-block linear representation
$\phi$ on a set $B$, then for every $x \in L$, $\phi$ restricted to the interval
$[\hatzero, x]$ can be restricted to a smaller base set to construct a
(faithful) representation $\phi_x$ over a finite base set.
\end{theorem}
\begin{proof}
\leavevmode

Consider the value $\phi(x)$.  Since $\phi$ is finite-block, it is
composed of disjoint finite blocks, which we call $B_i$ for $i$
in some index set $I$.
Every block of every $\phi(y)$ for $y \in
[\hatzero, x]$ is contained in some block $B_i$.
This means that $\phi$ can be decomposed into the sum of a set of (not
necessarily faithful) linear representations $\phi_i$ which are $\phi$
restricted to $B_i$:
\begin{align*}
\phi & = \bigoplus_{i \in I} \phi_i \\
B & = \bigcup_{i \in I} B_i
\end{align*}
Because $\phi$ is faithful,
for every one of the finite set of pairs $y, z \in [\hatzero, x]$ there is at
least one $i \in I$ for which $\phi(y) \neq \phi(z)$.  We call this
$i = i_{y,z}$ and we call the set of all $i_{y,z}$ (which may not be distinct)
$I_x$, which is finite.  Set
\begin{align*}
\phi_x & = \bigoplus_{i \in I_x} \phi_i \\
B_x & = \bigcup_{i \in I_x} B_i
\end{align*}
By construction, $\phi_x$ is a linear representation of $[\hatzero, x]$.
Because $I_x$ is finite and each $B_i$ is finite, $B_x$ is finite.
Because for every $y, z \in [\hatzero, x]$, $I_x$ contains the index
of a $\phi_i$ that distinguishes $y$ and $z$, $\phi_x$ is faithful.
\end{proof}

\begin{question} \label{q:finitary-block}
Does every finitary linear lattice have a finite-block representation?
\end{question}

\begin{theorem} {}[\ref{th:Dn-misc}]\cite{Haim1985a}*{sec.~2.0 Exam.~2.5}
A linear lattice satisfies all of the properties $D_n$ and so is
Arguesian and modular.
\end{theorem}

\begin{definition} \cite{Haim1985a}*{sec.~0.0}
Let $\textnormal{Eq } S$ denote the lattice of equivalence relations of
the set $S$ (which are subsets of $S \times S$).
Let $\rho: L \rightarrow \textnormal{Eq } S$ and $\rho^\prime: L^\prime
\rightarrow \textnormal{Eq } S^\prime$ be linear representations of
lattices $L$ and $L^\prime$.

If $S \cap S^\prime = \zeroslash$, the \emph{sum} $\rho \oplus \rho^\prime:
L \times L^\prime \rightarrow \textnormal{Eq } (S \cup S^\prime)$ is defined as
$\rho \oplus \rho^\prime((x, x^\prime)) = \rho(x) \cup
\rho^\prime(x^\prime)$.  If $L =
L^\prime$ we also refer to $(\rho \oplus \rho^\prime) \circ \Delta: L
\rightarrow \textnormal{Eq } (S \cup S^\prime)$,
where $\Delta(x) = (x, x)$ [the minimum,
discrete, or identity equivalence relation], as the
\emph{sum} of the representations $\rho$, $\rho^\prime$ of $L$.

The \emph{product} $\rho \otimes \rho^\prime: 
L \times L^\prime \rightarrow \textnormal{Eq } (S \times S^\prime)$ is
defined as
$\rho \otimes \rho^\prime((x, x^\prime)) =
\{ ((a,a^\prime), (b,b^\prime)) : (a, b) \in \rho(x) \TNand
(a^\prime,b^\prime) \in \rho^\prime(x^\prime) \}$.
If $L =
L^\prime$ we also refer to $(\rho \otimes \rho^\prime) \circ \Delta: L
\rightarrow \textnormal{Eq } (S \times S^\prime)$, as the
\emph{product} of the representations $\rho$, $\rho^\prime$ of $L$.

Sums and products
of arbitrary finite or infinite collections of representations are
defined analogously.
\end{definition}

\begin{theorem} \label{th:sum-prod}
The sum and product of linear representations are linear representations.
\end{theorem}

\begin{theorem} \label{th:linear-subd} \cite{Jons1959a}\cite{Haim1984a}*{Ch.~II}\linebreak[0]\cite{Haim1985a}*{sec.~1.0}\linebreak[0]\cite{NatPick1987}\linebreak[0]\cite{Nat2017}*{ch.~4}
The class of linear lattices is closed under direct products and
taking sublattices.
\end{theorem}
\begin{proof}
\leavevmode

Closure under sublattices is trivial.  Closure under direct products
is demonstrated by the sum of linear representations (and also by the
product of linear representations).
\end{proof}

\begin{question} \label{q:linear-variety} \cite{Jons1953b}*{sec.~5}\linebreak[0]\cite{Jons1959a}*{sec.~4 Prob.~1}\linebreak[0]\cite{Haim1984a}*{Ch.~II}\linebreak[0]\cite{Haim1985a}*{sec.~1.0}\linebreak[0]\cite{Nat1994}\linebreak[0]\cite{Nat2017}*{ch.~4}
Is the class of linear lattices closed under homomorphic images (which
by theorems~\ref{th:linear-subd} and~\ref{th:birk-variety} is
equivalent to being a variety and to being definable by
identities)?
\end{question}

\begin{question} \label{q:linear-self-dual} \cite{Haim1985a}*{sec.~1.0}
Is the class of linear lattices self-dual?
\end{question}

\begin{question} \label{q:linear-strictly-stronger}
Is the linear condition strictly stronger than the unbounded-order
Arguesian condition?
That is, are there unbounded-order Arguesian lattices that are not
linear?
If so, does there exist such a finite lattice?
(If not, then questions~\ref{q:linear-variety}
and~\ref{q:linear-self-dual} are true.)
\end{question}

\begin{theorem} {}[parallel to theorem~\ref{th:uoa-infinite}]\cite{Haim1987}*{sec.~4}\cite{Haim1991}*{sec.~3 Cor.~5}\cite{Nat1994}
The class of linear lattices is not definable by a finite set of
sentences.
\end{theorem}

\begin{theorem} \label{th:linear-sent} \cite{Jons1959a}*{sec.~3 Th.~2}\cite{Haim1985a}\cite{NatPick1987}\cite{Haim1991}*{sec.~1}
There is an infinite
system of universal Horn sentences characterizing linear lattices.
Specifically, in order for a lattice $A$ to be linear, it is necessary and
sufficient that the following condition be satisfied:

Suppose $n$ is a positive integer, $a_0, a_1, \cdots, a_{2n},
z \in A$ and, for each
positive integer $k \leq n$, $\phi(k)$ and $\psi(k)$ are natural
numbers with $\phi(k) \leq k$ and $\phi(k) \leq k$. Let
$$ b_{1,0,1} = b_{1,1,0} = a_0,\ b_{1,0,0} = b_{1,1,1} = z, $$
and for $k = 1,2, \cdots, n$ and $i, j = 0, 1, \cdots , k$ let
\begin{align*}
b_{k+1,i,j} & = \bigwedge_{p \leq k} (b_{k,i,p} \vee b_{k,p,j}), \\
b_{k+1,i,k+1} & = b_{k+1,k+1,i} = (b_{k+1,i,\phi(k)} \vee a_{2k-1})
\wedge (b_{k+1,i,\psi(k)} \vee a_{2k}), \\
b_{k+1.k+1,k+1} & = z.
\end{align*}
With these definitions, if $b_{k,\phi(k),\psi(k)} \leq a_{2k-1} \vee a_{2k}$
for $k = 1, 2, \cdots, n$, then $a_0 \leq b_{n+1, 0, 1}$.
\end{theorem}

\begin{question} \label{q:linear-implication}
Which of the sentences defined by theorem~\ref{th:linear-sent} are
implied by one or more of the laws $D_n$?
Conversely, which of the laws $D_n$ are implied by one or more of the
sentences of theorem~\ref{th:linear-sent}?
\end{question}

\begin{theorem} \cite{NatPick1987}*{Proof of Lem.~4}
\leavevmode

\begin{enumerate}
\item For a fixed lattice $L$, the type 1 representations of $L$ form an
elementary class\cite{WikiEl} of relational structures.
\item \cite{Jons1959a} A lattice $L$ is linear if
and only if every finitely generated sublattice of $L$ is linear.
\item If every countable sublattice of $L$ is linear, $L$ is.
\item \cite{Jons1959a} If a countable lattice $L$ is linear,
it has a type 1 representation with a countable base set.
\end{enumerate}
\end{theorem}

\begin{theorem} \cite{Haim1984a}*{Ch.~II}\cite{Haim1985a}*{sec.~1.0}
The class of linear lattices is closed under direct limits.
\end{theorem}
\begin{proof}
\leavevmode

Because the class is defined by implications, and every implication
depends on only finitely many hypotheses $P_i \leq Q_i$,
every implication that is satisfied by every lattice in a direct
system is satisfied by the direct limit of the direct system.
\end{proof}

\begin{question} \label{q:linear-direct-limit} \cite{Haim1984a}*{Ch.~II}\cite{Haim1985a}*{sec.~1.0}
It does not seem possible, however, straightforwardly to
construct a representation of a direct limit of linear lattices from
representations of the individual lattices.
Is there such a construction?
\end{question}

\paragraph{Normal subgroups of a group}

\begin{theorem} \label{th:normal-linear} \cite{Jons1953b}*{sec.~1}\cite{Nat2017}*{ch.~4}
The normal subgroup lattice of a
group $G$ has a natural linear representation $(X, F)$: take $X = G$ and
$F(N) = \{(x, y) \in G^2:  xy^{-1} \in N\}$. 
Thus the normal subgroup lattice of a group is linear, and so is any
sublattice of normal subgroups of a group.
\end{theorem}

\begin{theorem} \cite{Jons1954a}*{Th.~3.6}\cite{Herr2013}*{sec.~3}
There is a finite linear lattice (called $B^\prime$
in \cite{Jons1954a}) which is not
embeddable into the normal subgroup lattice of any group.
\end{theorem}

\begin{question} \label{q:normal-dual}
Is the class of sublattices of normal subgroups of a group self-dual?
\end{question}

\begin{question} \label{q:normal-sent} \cite{Jons1953b}*{sec.~5}
Can the class of sublattices of normal subgroups of a group be
specified by a set of sentences?  By a set of identities?  
By a finite set of sentences?  By a finite set of identities?
\end{question}

\paragraph{Subgroups of an Abelian group}

\begin{theorem}
Because all subgroups of Abelian groups are normal, every lattice of
subgroups of an Abelian group is also a lattice of normal subgroups of
the same group.
\end{theorem}

\begin{theorem} \label{th:abelian-normal-1} \cite{PalSzab1995}
The identity
$$ x_1 \wedge \{y_1 \vee [(x_2 \vee y_2) \wedge (x_3 \vee y_3) \wedge (x_4 \vee y_4)]\}
\leq [(p_{12} \vee p_{34}) \wedge (p_{13} \vee p_{24}) \wedge (p_{14} \vee p_{23})] \vee x_2 \vee y_3 \vee y_4 $$
where $p_{ij} = (x_i \vee y_j) \wedge (x_j \vee y_i)$, holds in the
subgroup lattice of every Abelian group but fails in the lattice of
normal subgroups of some finite group.
\end{theorem}

\begin{theorem} \label{th:abelian-normal-2} \cite{KissPal1998}*{sec.~1}
The lattice of normal subgroups of the three generator free
group $G$ in the group variety defined by the laws $x^4 - 1$ and
$x^2y - yx^2$ cannot
be embedded into the subgroup lattice of any Abelian group.
\end{theorem}

\begin{theorem}
Theorems~\ref{th:abelian-normal-1} and \ref{th:abelian-normal-2} each imply
that the property of being a sublattice of the
subgroups of an Abelian group is strictly stronger than the property
of being a sublattice of the normal subgroups of a general group.
Theorem~\ref{th:abelian-normal-1} implies this for finite groups.
\end{theorem}

\begin{question} \label{q:abelian-dual}
Are the set of lattices that are embeddable into the subgroups of an
Abelian group self-dual?
\end{question}

\begin{question} \label{q:abelian-sent} \cite{Jons1953b}*{sec.~5}
Can the class of sublattices of subgroups of an Abelian group be
specified by a set of sentences?  By a set of identities?  
By a finite set of sentences?  By a finite set of identities?
\end{question}

\paragraph{Subspaces of a skew vector space}

\begin{definition}
A \emph{skew vector space} is a module over a division ring.
\end{definition}

\begin{theorem} (the Veblen--Young theorem)
For any projective space in which Desargues' theorem holds, the
incidence lattice of the projective space is isomorphic to the
subspace lattice of a skew vector space.
\end{theorem}

\begin{theorem}
Because a skew vector space is Abelian with respect to the operator
$+$, its subspace lattice is the lattice of subgroups
of the vector space taken as an Abelian group.
\end{theorem}

\begin{theorem} \label{th:skew-dual}
The class of sublattices of subspaces of a skew vector space is self-dual.
\end{theorem}
\begin{proof}
\leavevmode

The map of a subspace $W$ of a skew vector space $V$ into the
annihilator $W^0$ of $W$, which is a subspace of the dual $V^*$ of $V$,
is order-reversing and embeds the lattice of subspaces of $V$ into the
lattice of subspaces of $V^*$.  (If $V$ is finite-dimensional, the
mapping is a dual-isomorphism.)
\end{proof}

\begin{question} \label{q:skew-abelian}
Are there lattices of subgroups of Abelian groups that cannot be
embedded into the lattice of subspaces of a skew vector space?
Of those lattices, are there any that are finite?
\end{question}

\begin{question} \label{q:skew-sent}
Can the class of sublattices of subspaces of skew vector spaces be
specified by a set of sentences?  By a set of identities?  
By a finite set of sentences?  By a finite set of identities?
\end{question}

\paragraph{Subspaces of a vector space}

\begin{theorem} (the Veblen--Young theorem)
For any projective space in which Desargues' theorem and Pappus'
hexagon theorem hold, the
incidence lattice of the projective space is isomorphic to the
subspace lattice of a vector space.
\end{theorem}

\begin{question} \label{q:vector-finite}
Is there a lattice embeddable into the subspaces of a skew
vector space which cannot be embedded into the subspaces of a vector
space?
If so, is there such a lattice that is finite?
\end{question}

\begin{remark} \cite{Day1981}
The incidence lattice of subspaces of the 3-dimensional skew vector
space over the quaternions (the (non-Pappian) projective plane over the
quaternions) can be embedded into the incidence lattice of subspaces
of the 6-dimensional vector space over the complex numbers (the
(Pappian) 5-dimensional projective space over the complex numbers).
\end{remark}

\begin{theorem} \label{th:vector-dual}
The class of sublattices of subspaces of a vector space is self-dual.
\end{theorem}
\begin{proof}
\leavevmode

The map of a subspace $W$ of a vector space $V$ into the
annihilator $W^0$ of $W$, which is a subspace of the dual $V^*$ of $V$,
is order-reversing and embeds the lattice of subspaces of $V$ into the
lattice of subspaces of $V^*$.  (If $V$ is finite-dimensional, the
mapping is a dual-isomorphism.)
\end{proof}

\begin{question} \label{q:vector-sent}
Can the class of sublattices of subspaces of vector spaces be
specified by a set of sentences?  By a set of identities?  
By a finite set of sentences?  By a finite set of identities?
\end{question}

\paragraph{Distributive}

\begin{theorem} \label{th:dist}
For a lattice $L$, the following conditions are equivalent:
\begin{enumerate}

\item \label{th:dist:11} \cite{Birk1967}*{ch.~I sec.~6 Th.~9}\cite{Gratz1996}*{Ch.~I sec.~4 Lem.~10}
$L$ satisfies
$$ x \wedge (y \vee z) = (x \wedge y) \vee (x \wedge z) $$

\item \label{th:dist:12} \cite{Birk1967}*{ch.~I sec.~6 Th.~9}\cite{Gratz1996}*{Ch.~I sec.~4 Lem.~10}
$L$ satisfies
$$ x \vee (y \wedge z) = (x \vee y) \wedge (x \vee z) $$

\item \label{th:dist:13} \cite{Gratz1996}*{Ch.~I sec.~4 Lem.~10}
$L$ satisfies
$$ (x \vee y) \wedge z \leq x \vee (y \wedge z) $$

\item \label{th:dist:14} \cite{Pal2001}*{sec.~2}
$L$ satisfies
$$ (x \vee y) \wedge (x \vee z) \wedge (y \vee z) =
(x \wedge y) \vee (x \wedge z) \vee (y \wedge z) $$

\item \label{th:dist:15} \cite{Birk1967}*{ch.~II sec.~7 Th.~13 Cor.}
$L$ satisfies
$$ a \wedge x = a \wedge y \TNand a \vee x = a \vee y \TNimplies x=y $$
(That is, relative complements be unique.)

\item \cite{Birk1967}*{ch.~II sec.~7 Th.~13}
$L$ is modular and does not contain the ``diamond lattice'' (which is called both
$M_3$ and $M_5$)\footnote{The older usage seems to be to call it $M_5$
after the number of elements.  The newer usage seems to be to call it
$M_3$ after the number of rank-1 elements, extending to $M_n$ being
the diamond lattice with $n$ rank-1 elements.}
as a sublattice:
\\
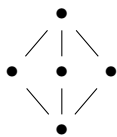
\begin{tikzcd}[column sep=0.5em,row sep=1em,every arrow/.append style={dash}]
  & \bullet \ar[dl] \ar[d] \ar[dr] \\
\bullet \ar[dr] & \bullet \ar[d] & \bullet \ar[dl] \\
  & \bullet
\end{tikzcd} \\
\end{enumerate}
\end{theorem}

\begin{definition}
A lattice satisfying any of the conditions of theorem~\ref{th:dist} is
called \emph{distributive}.
\end{definition}

\begin{theorem}
There are finite lattices that are the subspace lattice of a vector
space that are not distributive.
\end{theorem}
\begin{proof}
\leavevmode

Specifically, consider the subspace lattice of $PG(\mathbb{F}_2, 2)$, where
${<} v_1, v_2, \ldots {>}$ denotes the subspace spanned by $v_1, v_2, \ldots$:
\\
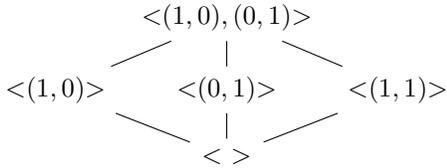
\begin{tikzcd}[column sep=0.5em,row sep=1em,every arrow/.append style={dash}]
  & {<}(1,0), (0,1){>} \ar[dl] \ar[d] \ar[dr] \\
{<}(1,0){>} \ar[dr] & {<}(0,1){>} \ar[d] & {<}(1,1){>} \ar[dl] \\
  & {<}\ {>}
\end{tikzcd} \\
\end{proof}

\begin{theorem}
The class of distributive lattices is defined by a finite number of
identities, specifically the axioms of lattices and any one of the
identities \ref{th:dist:11}, \ref{th:dist:12}, \ref{th:dist:13},
or \ref{th:dist:14} of theorem~\ref{th:dist}.
Because such set of identities is self-dual (in the case
of \ref{th:dist:13} or \ref{th:dist:14}), or dual to an equivalent set
of identities (in the case of \ref{th:dist:11} or \ref{th:dist:12}),
the class of distributive lattices is self-dual.
\end{theorem}

\begin{theorem} (Birkhoff's Representation Theorem)
\cite{Birk1967}*{ch.~III sec.~3 Th.~3}\cite{Stan2012}*{Prop. 3.4.3}
Let $L$ be a finite or finitary distributive lattice.  Let
$X$ be the poset of join-irreducible elements of $L$ (from which we exclude
$\hatzero$).
Then $L$ is isomorphic to $\boldsymbol{2}^X$ (the set of weakly order-preserving
functions from $X$ to $\boldsymbol{2}$) or equivalently, $J_f(X)$ (the
set of finite order ideals of $X$).
\end{theorem}

\paragraph{Complemented}

We are not primarily interested in complemented lattices, but we
include this interesting result:

\begin{theorem} \cite{Jons1954a}*{sec.~1 Th.~2.14}\cite{Jons1954c}
Within complemented modular lattices, being Arguesian is equivalent to
being isomorphic to a lattice of subspaces of a skew vector space.
Consequently, for complemented modular lattices, all properties in the
hierarchy we have been studying from Arguesian to being isomorphic to
a lattice of subspaces of a skew vector space are equivalent.
\end{theorem}

\vfill   % The next thing is a table so allow this page to have trailing space

\section{Summary of properties} \label{sec:summary-chart}

{
% Add some space to the top and bottom of the cells.
\renewcommand{\arraystretch}{1.5}
% Make the type smaller so I can fit more in.
\small
% The column width.
\newcommand{\cwd}{0.125\textwidth}

\begin{tabular}{|p{\cwd}|p{\cwd}|p{\cwd}|p{\cwd}|p{\cwd}|p{\cwd}|p{\cwd}|}
\hline
\emph{Property} &
\emph{Strictly stronger than the preceding property?} \raggedright &
\emph{Strictly stronger on finite lattices?} \raggedright &
\emph{Self-dual?} &
\emph{Defined by sentences?} \raggedright &
\emph{Defined by identities?} \raggedright &
\emph{Defined by a finite number of sentences/\linebreak[0]identities?}
\raggedright \tabularnewline  % // fails here
\hline
General lattice & --- & --- & yes & yes & yes & yes \\
\hline
Modular & yes & yes & yes & yes & yes & yes \\
\hline
Arguesian & yes & yes & yes & yes & yes & yes \\
\hline
Higher-order Arguesian &
some are, unknown whether all \raggedright &
some are, unknown whether all \raggedright & yes & yes & yes & yes \\
\hline
Unbounded-order Arguesian \raggedright & yes & yes & yes & yes & yes & no \\
\hline
Linear & unknown & unknown & unknown & yes & unknown & no \\
\hline
Sublattices of normal subgroups of a group \raggedright &
yes & yes & unknown & unknown & unknown & unknown \\
\hline
Sublattices of subgroups of an Abelian group \raggedright &
yes & yes & unknown & unknown & unknown & unknown \\
\hline
Subspaces of a skew vector space \raggedright &
unknown & unknown & yes & unknown & unknown & unknown \\
\hline
Subspaces of a vector space \raggedright &
unknown & unknown & yes & unknown & unknown & unknown \\
\hline
Distributive & yes & yes & yes & yes & yes & yes \\
\hline
\end{tabular}
}

\section{Summary of open questions} \label{sec:summary-qs}

We collect the open questions here.

% Insert the list of questions after here.

\par \hbox{}

% arguesian.tex:1707
\noindent \textbf{Question \ref{q:finite-algebra}\@.} {\itshape   \cite{Pal2001}*{sec.~3 Prob.~3.4}
It is unknown whether for every finite lattice $L$ there exists a
finite algebra $\boldsymbol{A}$ such that
$\textnormal{Con } \boldsymbol{A} \cong L$.
Is this true?

} \par \hbox{}

% arguesian.tex:1751
\noindent \textbf{Question \ref{q:finite-type-3}\@.} {\itshape  
Given theorems~\ref{th:finite-lattice-finite-part}
and~\ref{th:all-type-3}, is it true that every finite lattice has a
type-3 representation over a finite base set?

} \par \hbox{}

% arguesian.tex:1883
\noindent \textbf{Question \ref{q:finite-type-2}\@.} {\itshape  
Given theorems~\ref{th:finite-lattice-finite-part}
and~\ref{th:modular-type-2}, is it true that every finite modular lattice has a
type-2 representation over a finite base set?

} \par \hbox{}

% arguesian.tex:2215
\noindent \textbf{Question \ref{q:Dn-strictly-stronger}\@.} {\itshape   \cite{Haim1987}*{sec.~4}
It is expected that $D_{n+1}$ is strictly stronger than $D_n$ for all
$n \geq 1$.  Is this true?  Is this true in the class of finite lattices?

} \par \hbox{}

% arguesian.tex:2302
\noindent \textbf{Question \ref{q:finite-type-1}\@.} {\itshape   \cite{Wor2024a}
Given theorem~\ref{th:finite-lattice-finite-part},
is it true that every finite linear lattice has a
type-1 representation over a finite base set?

} \par \hbox{}

% arguesian.tex:2353
\noindent \textbf{Question \ref{q:finitary-block}\@.} {\itshape  
Does every finitary linear lattice have a finite-block representation?

} \par \hbox{}

% arguesian.tex:2411
\noindent \textbf{Question \ref{q:linear-variety}\@.} {\itshape   \cite{Jons1953b}*{sec.~5}\linebreak[0]\cite{Jons1959a}*{sec.~4 Prob.~1}\linebreak[0]\cite{Haim1984a}*{Ch.~II}\linebreak[0]\cite{Haim1985a}*{sec.~1.0}\linebreak[0]\cite{Nat1994}\linebreak[0]\cite{Nat2017}*{ch.~4}
Is the class of linear lattices closed under homomorphic images (which
by theorems~\ref{th:linear-subd} and~\ref{th:birk-variety} is
equivalent to being a variety and to being definable by
identities)?

} \par \hbox{}

% arguesian.tex:2418
\noindent \textbf{Question \ref{q:linear-self-dual}\@.} {\itshape   \cite{Haim1985a}*{sec.~1.0}
Is the class of linear lattices self-dual?

} \par \hbox{}

% arguesian.tex:2422
\noindent \textbf{Question \ref{q:linear-strictly-stronger}\@.} {\itshape  
Is the linear condition strictly stronger than the unbounded-order
Arguesian condition?
That is, are there unbounded-order Arguesian lattices that are not
linear?
If so, does there exist such a finite lattice?
(If not, then questions~\ref{q:linear-variety}
and~\ref{q:linear-self-dual} are true.)

} \par \hbox{}

% arguesian.tex:2459
\noindent \textbf{Question \ref{q:linear-implication}\@.} {\itshape  
Which of the sentences defined by theorem~\ref{th:linear-sent} are
implied by one or more of the laws $D_n$?
Conversely, which of the laws $D_n$ are implied by one or more of the
sentences of theorem~\ref{th:linear-sent}?

} \par \hbox{}

% arguesian.tex:2492
\noindent \textbf{Question \ref{q:linear-direct-limit}\@.} {\itshape   \cite{Haim1984a}*{Ch.~II}\cite{Haim1985a}*{sec.~1.0}
It does not seem possible, however, straightforwardly to
construct a representation of a direct limit of linear lattices from
representations of the individual lattices.
Is there such a construction?

} \par \hbox{}

% arguesian.tex:2515
\noindent \textbf{Question \ref{q:normal-dual}\@.} {\itshape  
Is the class of sublattices of normal subgroups of a group self-dual?

} \par \hbox{}

% arguesian.tex:2519
\noindent \textbf{Question \ref{q:normal-sent}\@.} {\itshape   \cite{Jons1953b}*{sec.~5}
Can the class of sublattices of normal subgroups of a group be
specified by a set of sentences?  By a set of identities?  
By a finite set of sentences?  By a finite set of identities?

} \par \hbox{}

% arguesian.tex:2557
\noindent \textbf{Question \ref{q:abelian-dual}\@.} {\itshape  
Are the set of lattices that are embeddable into the subgroups of an
Abelian group self-dual?

} \par \hbox{}

% arguesian.tex:2562
\noindent \textbf{Question \ref{q:abelian-sent}\@.} {\itshape   \cite{Jons1953b}*{sec.~5}
Can the class of sublattices of subgroups of an Abelian group be
specified by a set of sentences?  By a set of identities?  
By a finite set of sentences?  By a finite set of identities?

} \par \hbox{}

% arguesian.tex:2599
\noindent \textbf{Question \ref{q:skew-abelian}\@.} {\itshape  
Are there lattices of subgroups of Abelian groups that cannot be
embedded into the lattice of subspaces of a skew vector space?
Of those lattices, are there any that are finite?

} \par \hbox{}

% arguesian.tex:2605
\noindent \textbf{Question \ref{q:skew-sent}\@.} {\itshape  
Can the class of sublattices of subspaces of skew vector spaces be
specified by a set of sentences?  By a set of identities?  
By a finite set of sentences?  By a finite set of identities?

} \par \hbox{}

% arguesian.tex:2620
\noindent \textbf{Question \ref{q:vector-finite}\@.} {\itshape  
Is there a lattice embeddable into the subspaces of a skew
vector space which cannot be embedded into the subspaces of a vector
space?
If so, is there such a lattice that is finite?

} \par \hbox{}

% arguesian.tex:2648
\noindent \textbf{Question \ref{q:vector-sent}\@.} {\itshape  
Can the class of sublattices of subspaces of vector spaces be
specified by a set of sentences?  By a set of identities?  
By a finite set of sentences?  By a finite set of identities?

} \par \hbox{}

% Insert the list of questions before here.

% Push the e-mail address down a bit.
\vspace{3em}

\end{document}